\documentclass[a4paper,11pt]{article}
\usepackage[utf8]{inputenc}

\usepackage{amsthm}
\usepackage{amssymb}
\usepackage{amsmath}
\usepackage{mathtools}
\usepackage[mathscr]{euscript}
\usepackage[normalem]{ulem}
\usepackage{stmaryrd}

\usepackage[nocompress]{cite} 

\usepackage[dvipsnames]{xcolor}
\usepackage[xcolor]{changebar}
\usepackage{pgf, tikz, pgfplots}
\pgfplotsset{compat=1.17}
\usepackage{appendix}
\usetikzlibrary{calc, arrows, shapes, positioning, patterns, angles, quotes, intersections, through, backgrounds}
\usetikzlibrary{decorations.markings}
\graphicspath{ {figures/} }

\usepackage[titles]{tocloft}
\usepackage{fancyhdr}
\usepackage{orcidlink}
\usepackage{hyperref}
\hypersetup{    
}

\usepackage[inline]{enumitem}
\newcommand{\enumlabelformat}{\roman}
\newcommand{\enumlabelfont}[1]{#1}
\newlength{\thelabelsep}
\setlist{labelsep=\thelabelsep}
\setlist[enumerate,1]{font=\enumlabelfont,label=(\enumlabelformat*),leftmargin=2.5em}
\setlist[itemize]{leftmargin=2.5em,label=$-$}

\newcounter{inlineenum}
\renewcommand{\theinlineenum}{\enumlabelformat{inlineenum}}

\usepackage{multirow}

\usepackage{zref-clever}

\let\oldnewtheorem\newtheorem

\RenewDocumentCommand{\newtheorem}{momo}{%
  \IfValueTF{#2}{%
    \AddToHook{env/#1/begin}{%
      \zcsetup{countertype={#2=#1}}}%
      \zcRefTypeSetup{#1}{
Name-sg = #3 ,%
      }%
    \oldnewtheorem{#1}[#2]{#3}%
  }{%
    \AddToHook{env/#1/begin}{
      \zcsetup{countertype={#1=#1}}}%
    \zcRefTypeSetup{#1}{
Name-sg = #3 ,%
      }%
    \IfValueTF{#4}{%
      \oldnewtheorem{#1}{#3}[#4]%
    }{%
      \oldnewtheorem{#1}{#3}%
    }%
  }%
}

\newcommand{\cref}[1]{\zcref{#1}}
\newcommand{\Cref}[1]{\zcref[S]{#1}}

\numberwithin{equation}{section}

\theoremstyle{definition}
\newtheorem{dfn}{Definition}[section]

\newtheorem{thm}[dfn]{Theorem}
\newtheorem{prop}[dfn]{Proposition}
\newtheorem{lem}[dfn]{Lemma}

\newtheorem{rem}[dfn]{Remark}

\newenvironment{acknowledgements}{%
\begin{abstract}
}{%
\end{abstract}
}

\fancypagestyle{firstpage}
{
\fancyhead[L]{}    
\fancyhead[R]{DMUS-MP-23/14}
}

\DeclareMathOperator{\id}{id}

\DeclareMathOperator{\Int}{Int}

\DeclareMathOperator{\arsinh}{arsinh}

\let\epsilon\varepsilon 
\let\phi\varphi

\newcommand{\R}{\mathbb{R}}
\newcommand{\N}{\mathbb{N}}

\newcommand{\LpLS}{Lo\-rentz\-ian pre-length space }
\newcommand{\LpLSn}{Lo\-rentz\-ian pre-length space}

\newcommand*{\bx}{\bar{x}}
\newcommand*{\by}{\bar{y}}
\newcommand*{\bz}{\bar{z}}

\newcommand*{\bp}{\bar{p}}
\newcommand*{\bq}{\bar{q}}

\newcommand*{\ba}{\bar{a}}
\newcommand*{\bb}{\bar{b}}
\newcommand*{\bO}{\bar{O}}

\newcommand*{\bT}{\bar{T}}
\newcommand*{\bC}{\bar{C}}
\newcommand*{\tC}{\tilde{C}}

\newcommand*{\tO}{\tilde{O}}
\newcommand{\tp}{\tilde{p}}

\newcommand{\hx}{\hat{x}}
\newcommand{\hy}{\hat{y}}
\newcommand{\hz}{\hat{z}}
\newcommand{\hp}{\hat{p}}

\newcommand{\tz}{\tilde{z}}

\makeatletter
\let\save@mathaccent\mathaccent
\newcommand*\if@single[3]{%
\setbox0\hbox{${\mathaccent"0362{#1}}^H$}%
\setbox2\hbox{${\mathaccent"0362{\kern0pt#1}}^H$}%
\ifdim\ht0=\ht2 #3\else #2\fi
}
\newcommand*\rel@kern[1]{\kern#1\dimexpr\macc@kerna}
\newcommand*\widebar[1]{\@ifnextchar^{{\wide@bar{#1}{0}}}{\wide@bar{#1}{1}}}
\newcommand*\wide@bar[2]{\if@single{#1}{\wide@bar@{#1}{#2}{1}}{\wide@bar@{#1}{#2}{2}}}
\newcommand*\wide@bar@[3]{%
\begingroup
\def\mathaccent##1##2{%
\let\mathaccent\save@mathaccent
\if#32 \let\macc@nucleus\first@char \fi
\setbox\z@\hbox{$\macc@style{\macc@nucleus}_{}$}%
\setbox\tw@\hbox{$\macc@style{\macc@nucleus}{}_{}$}%
\dimen@\wd\tw@
\advance\dimen@-\wd\z@
\divide\dimen@ 3
\@tempdima\wd\tw@
\advance\@tempdima-\scriptspace
\divide\@tempdima 10
\advance\dimen@-\@tempdima
\ifdim\dimen@>\z@ \dimen@0pt\fi
\rel@kern{0.6}\kern-\dimen@
\if#31
\overline{\rel@kern{-0.6}\kern\dimen@\macc@nucleus\rel@kern{0.4}\kern\dimen@}%
\advance\dimen@0.4\dimexpr\macc@kerna
\let\final@kern#2%
\ifdim\dimen@<\z@ \let\final@kern1\fi
\if\final@kern1 \kern-\dimen@\fi
\else
\overline{\rel@kern{-0.6}\kern\dimen@#1}%
\fi
}%
\macc@depth\@ne
\let\math@bgroup\@empty \let\math@egroup\macc@set@skewchar
\mathsurround\z@ \frozen@everymath{\mathgroup\macc@group\relax}%
\macc@set@skewchar\relax
\let\mathaccentV\macc@nested@a
\if#31
\macc@nested@a\relax111{#1}%
\else
\def\gobble@till@marker##1\endmarker{}%
\futurelet\first@char\gobble@till@marker#1\endmarker
\ifcat\noexpand\first@char A\else
\def\first@char{}%
\fi
\macc@nested@a\relax111{\first@char}%
\fi
\endgroup
}
\makeatother

\newcommand{\lm}[1]{\mathbb{L}^2(#1)}
\newcommand{\ma}{\ensuremath{\measuredangle}}

\newcommand{\mb}[1]{\mathbb{#1}}

\DeclarePairedDelimiterX\Set[1]{\lbrace}{\rbrace}%
{  #1 }

\title{Reshetnyak Majorisation and discrete upper curvature bounds for Lorentzian length spaces}

\author{Tobias Beran,\footnote{\href{mailto:tobias.beran@univie.ac.at}{tobias.beran@univie.ac.at}, Department of Mathematics, University of Vienna, Oskar-Morgenstern-Platz 1, 1090 Vienna, Austria. \,\orcidlink{0000-0002-2813-0099}} \:
Felix Rott\footnote{\href{mailto:frott@sissa.it}{frott@sissa.it}, SISSA, Via Bonomea 265, 34136 Trieste, Italy. \,\orcidlink{0000-0001-7314-0889}} \\ 
}

\begin{document}
\maketitle

\begin{abstract}
We present an analogue to the Majorisation Theorem of Reshetnyak in the setting of Lorentzian length spaces with upper curvature bounds: given two future-directed timelike rectifiable curves $\alpha$ and $\beta$ with the same endpoints in a \LpLS $X$, there exists a convex region in $\lm{K}$ bounded by two future-directed causal curves $\bar \alpha$ and $\bar \beta$ with the same endpoints and a 1-anti-Lipschitz map from that region into $X$ such that $\bar \alpha$ and $\bar \beta$ are respectively mapped $\tau$-length-preservingly onto $\alpha$ and $\beta$.
A special case of this theorem leads to an interesting characterisation of upper curvature bounds via four-point configurations which is truly suitable for a discrete setting. 
\bigskip

\medskip

\end{abstract}

\tableofcontents 

\section{Introduction}
It has been less than a decade since \cite{KS18} suggested the concept of Lorentzian length spaces and thus gave rise to the field of modern synthetic Lorentzian geometry. 
Originally motivated by low regularity phenomena in Lorentzian geometry and general relativity, and encouraged by the incredible impact that metric geometry had on Riemannian geometry, the field continues to develop and evolve at a very impressive rate. 
This work aims to add to this endeavour by giving a Lorentzian version of an important theorem about CAT($k$) spaces and using this to suggest a formulation for curvature bounds from above which is applicable in discrete spaces.

In positive signature, the original result \cite{Res68} is a fundamental theorem of spaces with upper curvature bounds, see \cite{AKP24} for a more modern treatment. 
It can be summarised as follows. 
Given a closed rectifiable curve $\gamma$ in a CAT($k$) space $X$, there exist: a closed curve $\bar \gamma$ in the model space $M_2(k)$ which bounds a convex region; and a 1-Lipschitz map from this region into $X$, such that the two curves are mapped onto each other in a length-preserving way. 

In the original formulation of \cite{Res68}, it is said that $\bar \gamma$ \emph{majorises} $\gamma$, hence also the name Majorisation Theorem for this result. 
Via Kolmogorov's Principle -- relating 1-Lipschitz maps to the non-increasement of area -- it can be seen as an isoperimetric result as well \cite{BH99}. 
See \cite{isoperimetricLLS} for similarly spirited results in the Lorentzian setting. 

The Lorentzian version of the theorem we propose in this work is as follows. 

\begin{thm}[Reshetnyak Majorisation, Lorentzian version]
Let $X$ be a strongly causal \LpLS with curvature bounded above by $K\in\mb R$ and let $U$ be a ($\leq K$)-comparison neighbourhood. 
Let $O \ll z$ with $\tau(O,z)<D_K$ in $U$ and let $\alpha$ and $\beta$ be two future-directed timelike rectifiable curves from $O$ to $z$ in $U$ forming a timelike loop $C$. 
Then there exists a causal loop $\tC$ in $\lm{K}$ bounding a convex region and a long map $f: R(\tC) \to U$ such that $\tC$ is mapped $\tau$-length preservingly onto $C$.
\end{thm}

On the one hand, this result is one of many bricks in the tall wall that is supporting the fundamental structure of the theory in general. 
On the other hand, the Majorisation Theorem carries with it a variety of curious and interesting applications. 
For example, it leads to a characterisation of upper curvature bounds which are applicable in discrete spaces, and it has been used, e.g., in \cite{KS93}. 
In the Lorentzian setting, discrete curvature bounds may be impactful for causal set theory, a discrete approach to quantum gravity \cite{BLMS87}. 
Other possible consequences include a lower bound on the length of causal curves and a Kirszbraun type theorem for anti-Lipschitz maps and will be discussed later on. 

\section{Preliminaries}
We refer to \cite{KS18, BKR24} for basic concepts about Lorentzian (pre-)length spaces, which we will in general simply denote by $X$. 
The metric $d$ on $X$ will never be used in our work, only the underlying topology plays a role. 
Thus, any result presented in this work retains its validity in any other setting of nonsmooth Lorentzian geometry, such as, e.g., metric spacetimes \cite{McC24, BM23+, Octet}, (bounded) Lorentzian metric spaces \cite{MS24, BMS24+} or weak and almost Lorentzian length spaces \cite{Mue23}. 
Some of these works assign the value $-\infty$ to pairs of points which are not causally related, which we will not do. 
However, wherever this is done, also the positive part of the time separation function is introduced and considered, hence in such a setting our results should be read with said positive part. 
We will denote the 2-dimensional Lorentzian model space of constant curvature $K$ by $\lm{K}$ and its (finite) diameter by $D_K$, cf.\ \cite[Definition 2.6]{BNR25}. 
In similarly standardised fashion, given $x \in X$, we decorate associated model space points in comparison triangles with a bar and those in four-point comparison configurations with a hat, i.e.\, $x, \bx, \hx$. 
This together with the fact that $K$ is fixed allows us to denote any time separation with the letter $\tau$ without specifying to which space it belongs. 
The points between which $\tau$ is measured will give all context to avoid any danger of confusion. 
Unless explicitly stated otherwise, we will assume any triangle $\Delta(x,y,z)$ to satisfy size bounds, i.e., $\tau(x,z) < D_K$.
Moreover, we will denote a geodesic between two points $x \ll y$ by $[x,y]$, and unless otherwise mentioned assume it to be parametrised in $[0,1]$ with constant speed. 
In spaces with upper curvature bounds, geodesics between points inside a comparison neighbourhood are unique, cf.\ \cite[Theorem 4.7]{BNR25}, so there is not even any ambiguity about this notation.
We may also use the notation $[x,y](t)$ to denote the point on $[x,y]$ at parameter $t \in [0,1]$. 
For the sake of completion, let us briefly recall the definition of triangle comparison. 

\begin{dfn}[Curvature bounds by triangle comparison]
\label{def: triangle comparison}
Let $X$ be a \LpLSn. An open subset $U$ is called a $(\leq K)$-comparison neighbourhood in the sense of \emph{triangle comparison} if:
\begin{enumerate}
\item $\tau$ is continuous on $(U\times U) \cap \tau^{-1}([0,D_K))$, and this set is open. 
\item $U$ is $D_K$-geodesic, i.e.\ for all $x \ll y$ in $U$ with $\tau(x,y) < D_K$ there exists a geodesic between those points inside $U$.
\item \label{TLCB.item3} Let $\Delta (x,y,z)$ be a timelike triangle in $U$, with $p,q$ two points on the sides of $\Delta (x,y,z)$. 
Let $\Delta(\bar{x}, \bar{y}, \bar{z})$ be a comparison triangle in $\lm{K}$ for $\Delta (x,y,z)$ and $\bar{p},\bar{q}$ comparison points for $p$ and $q$, respectively. 
Then 
\begin{equation}
\label{eq: timelike triangle comparison inequality}
\tau(p,q) \geq \tau(\bar{p}, \bar{q}) \, .
\end{equation}
\end{enumerate}
$X$ has curvature bounded above by $K$ in the sense of \emph{triangle comparison} if it is covered by such neighbourhoods.
\end{dfn}

Note that within a $(\leq K)$-comparison neighbourhood, $\bar{p} \ll \bar{q}$ implies $p \ll q$. 

\begin{dfn}[Curvature bounds by strict\footnote{\Cref{def: triangle comparison} and \Cref{def: strict triangle comparison} were respectively called timelike triangle comparison and strict causal triangle comparison in \cite{BKR24}. In this work we decided to drop the causality specifier.} triangle comparison]
\label{def: strict triangle comparison}
Let $X$ be a \LpLSn. An open subset $U$ is called a $(\leq K)$-comparison neighbourhood in the sense of \emph{strict triangle comparison} if:
\begin{enumerate}
\item $\tau$ is continuous on $(U\times U) \cap \tau^{-1}([0,D_K))$, and this set is open. 
\item $U$ is $D_K$-geodesic.
\item \label{TLCB.item3} Let $\Delta (x,y,z)$ be an admissible causal triangle\footnote{Recall that an \emph{admissible causal triangle} is a triangle (satisfying size bounds) in which one side is allowed to be null, cf.\ \cite[Remark 2.1]{BKR24}. Further note that it is not necessary to assume there exists a curve between the vertices which are null related (any of which would necessarily realise the distance), since we never choose (comparison) points on a null side. } in $U$, with $p,q$ two points on timelike sides of $\Delta (x,y,z)$. 
Let $\Delta(\bar{x}, \bar{y}, \bar{z})$ be a comparison triangle in $\lm{K}$ for $\Delta (x,y,z)$ and $\bar{p},\bar{q}$ comparison points for $p$ and $q$, respectively. 
Then 
\begin{equation}
\label{eq: strict triangle comparison inequality}
\tau(p,q) \geq \tau(\bar{p}, \bar{q}) \, \text{and}\, \bar{p}\leq\bar{q}\Rightarrow p\leq q \, .
\end{equation}
\end{enumerate}
$X$ has curvature bounded above by $K$ in the sense of \emph{strict triangle comparison} if it is covered by such neighbourhoods.
\end{dfn}

\begin{rem}[One-sided comparison]
\label{rem: one-sided comparison}
Further, recall that it is sufficient to require \eqref{eq: timelike triangle comparison inequality} only for pairs of points where one of them is a vertex of the triangle and the other lies on the opposite side. 
This condition is known as \emph{one-sided triangle comparison}, cf.\ \cite[Definition 3.2]{BKR24}. 
The same is true for \Cref{def: strict triangle comparison}, where the opposite side has to be timelike. 
Technically, this formulation was not introduced in \cite{BKR24}, but the obvious name would be strict one-sided comparison. 
Its equivalence to \Cref{def: strict triangle comparison} follows in the exact same spirit of \cite[Proposition 3.3]{BKR24}. 
\end{rem}

Perhaps the most powerful tool in Lorentzian triangle comparison theory is the so-called \emph{Law of Cosines Monotonicity}. 
In contrast to the metric setting, \emph{all} side lengths exhibit monotonous behaviour when an angle is changed while keeping the length of the other two sides fixed and vice versa. 
More precisely, we have the following, cf.\ \cite[Lemma 2.4 \& Remark 2.5]{BS23}. 

\begin{prop}[Law of Cosines (Monotonicity)]
\label{prop: law of cosines monotonicity}
Let $p, q, r \in \lm{K}$ form an admissible causal triangle (not necessarily in this order). 
Let $a = \max\{\tau (p, q), \tau (q, p)\} > 0$, $b = \max\{\tau (q, r), \tau (r, q)\} > 0$ and $c = \max \{\tau (p, r), \tau (r, p)\}$. 
Let $\omega=\ma_q(p,r)$ be the hyperbolic angle at $q$ and let $\sigma$ be its sign. 
Let $s=\sqrt{|K|}$. 
Then we have: 
\begin{align*}
a^2 + b^2 & = c^2 - 2ab\sigma \cosh(\omega) & K=0 \, , \\
\cos(sc) & = \cos(sa) \cos(sb) - \sigma \cosh(\omega) \sin(sa) \sin(sb) & K < 0 \, , \\
\cosh(sc) & = \cosh(sa) \cosh(sb) + \sigma \cosh(\omega) \sinh(sa) \sinh(sb) & K > 0 \, .
\end{align*}
In particular, fixing two sides and varying the third, $\omega$ is a strictly increasing function in the longest side and a strictly decreasing function in the other two sides. 
\end{prop}

The original Majorisation Theorem deals with closed curves. 
Clearly, a closed causal curve is a priori an undesired -- and in reasonably behaved spaces even impossible -- object. 
We get around this by considering two, say, future-directed causal curves with the same endpoints. 
This can be regarded as a loop where one of the curves is traversed backwards.

\begin{dfn}[Causal closed curves and length preserving maps]
Let $X$ be a \LpLS and let $x,y \in X$ with $x \ll y$. 
Let $\alpha$ and $\beta$ be two future-directed causal curves between $x$ and $y$.
\begin{enumerate}
\item We refer to the pair $(\alpha, \beta)$ as a \emph{causal loop} 
and may denote it by $C$. 
If both curves are timelike then we call $C$ a \emph{timelike loop}. 
\item The \emph{length} of a causal loop $C=(\alpha, \beta)$ is $L(C) \coloneqq L_{\tau}(\alpha) + L_{\tau}(\beta)$. 
\item Two causal loops $C=(\alpha, \beta), C'= (\alpha', \beta')$ are \emph{$\tau$-length-isometric} if $L_\tau(\alpha|_{[s,t]})=L_\tau(\alpha'|_{[s,t]})$ and $L_\tau(\beta|_{[s,t]})=L_\tau(\beta'|_{[s,t]})$ holds for all $s < t$. 
If there exists a mapping $f$ such that $C=(\alpha, \beta)$ and $C''=(f \circ \alpha, f \circ \beta)$ are $\tau$-length-isometric, then we say that $C$ is mapped\footnote{In his original work, Reshetnyak called such closed curves \emph{equilongal}.} \emph{$\tau$-length preservingly} onto $C''$. 
\item The \emph{region} associated to a causal loop $C$ in $\lm{K}$ is the closure of the connected component of the set $\lm{K}\setminus(\alpha\cup\beta)$ that is contained in $J(x,y)$. 
It is denoted by $R(C)$. 
Note that we will always consider regions to be equipped with their intrinsic time separation. 
If $R(C)$ is convex, then the intrinsic time separation agrees with the restricted one. 
\end{enumerate}
\end{dfn}

Next, we introduce the concept of anti-Lipschitz maps. 
In essence, one reverses the inequality in the Lipschitz condition and replaces $d$ by $\tau$. 

\begin{dfn}[Anti-Lipschitz maps]
Let $X,Y$ be \LpLSn s and let $f:U\subseteq X \to Y$ be a map on any subset. 
$f$ is called \emph{anti-Lipschitz} if there exists $K>0$ such that $K\tau_X(x,y) \leq \tau_Y(f(x),f(y))$ for all $x,y \in U$. 
The largest such $K$ is called the \emph{anti-Lipschitz constant} of $f$. 
We call $f$ \emph{strongly anti-Lipschitz} if it is anti-Lipschitz and $\leq$-preserving\footnote{Observe that this terminology is conflicting with \cite{BR24}, where $\leq$-preserving meant that the converse implication is in place as well. 
We have since adopted the terminology of splitting these two properties and calling them $\leq$-preserving and $\leq$-reflecting, respectively.}, i.e., $x \leq_X y \Rightarrow f(x) \leq_Y f(y)$.
In analogy to 1-Lipschitz maps in metric geometry being called short, we call $f$ \emph{long} if $\tau_X(x,y)\leq\tau_Y(f(x),f(y))$ for all $x,y \in U$, i.e.\ it is $1$-anti-Lipschitz. 
Finally, we call $f$ \emph{strongly long} if it is long and $\leq$-preserving. 
\end{dfn}

If one restricts the target space to be $\R$ considered as the `Minkowski line' $\R^{1,0}$, then one arrives at the concept of \emph{steep} functions as in, e.g., \cite[Definition 3.4]{Octet}. 
In this way, anti-Lipschitzness provides a generalisation of steepness. 

\section{Majorisation}

In this section we prove a Lorentzian version of the Majorisation Theorem. 
At first, we will give a very elementary definition, mostly to make formulations down the line more easily understandable. 

\begin{dfn}[Lorentzian spheres]
Let $X$ be a \LpLSn. 
By a \emph{future-directed hyperbola with radius $r$ and centre $x$} we mean the set of points $\{ y \in X \mid \tau(x,y) =r \}$. 
We will denote it by $H_r^+(x)$. 
Analogously, we define a \emph{past-directed hyperbola}. 
\end{dfn}

We start with collecting some initially helpful result. 
The following lemma is elementary, we state it mostly to make referencing to its argument in the following proofs more convenient. 

\begin{lem}[Maximising distances in model space]
\label{lem: Maximising distances in model space}
Let $p \in \lm{K}$. 
\begin{enumerate}
\item Let $z_1, z_2 \in H^+_r(p)$ with $z_1 \neq x_2$. 
Consider the (timelike) angle bisector $L$ of the angle $\ma_p(z_1,z_2)$. 
Suppose that $y \in \lm{K}$ is in the interior of the half space generated by $L$ that contains $z_2$. 
Then $\max\{\tau(y,z_1),\tau(z_1,y)\} < \max\{\tau(y,z_2),\tau(z_2,y)\}$. 

\item If $y_0 \in H^+_{r'}(p)$ with $r' < r$ and $z_0 \in H^+_r(p)$, then
$\tau(y_0,z_0)$ is strictly monotonically decreasing in $\ma_p(y_0,z_0)$ (or equivalently in any other angle). 
In particular, $\tau(y_0,z_0)=\max\{\tau(y,z) \mid y \in H^+_{r'}(p), z \in H^+_r(p)\}$ if and only if $y_0 \in [p,z_0]$. 
\end{enumerate}
\end{lem}
\begin{proof}
(i) We will only treat the case $K=0$, the other cases follow with the exact same logic. 
By applying a suitable Lorentz transformation, we can assume without loss of generality that $p=0$ and $z_1,z_2$ lie on the same horizontal axis, i.e., $z_1=(t_z,-x_z)$ and $z_2=(t_z,x_z)$, with, say, $x_z>0$. 
Then by assumption we also have $x_y > 0$ for $y=(t_y,x_y)$. 
Thus we simply compute
\begin{align*}
\max\{\tau(y,z_1),\tau(z_1,y)\} & = (t_z-t_y)^2-(x_z-x_y)^2  \\
& < (t_z-t_y)^2-(-x_z-x_y)^2 = \max\{\tau(y,z_2),\tau(z_2,y)\} \, .
\end{align*}
Note that the strict inequality implies the preservation of any causal relation, i.e., e.g., $y \leq z_1 \Rightarrow y \leq z_2$. In fact, even $y \leq z_1 \Rightarrow y \ll z_2$ holds true. 

(ii) This follows immediately with an application of \Cref{prop: law of cosines monotonicity}. 
Note that $\ma_p(y_0,z_0)=0$ if and only if $y_0 \in [p,z_0]$. 
\end{proof}

A past-directed version of \Cref{lem: Maximising distances in model space} follows analogously. 
The next lemma makes can be regarded as a `filled in' version of Alexandrov's Lemma and makes use of an elementary construction that we will define now. Towards readability, it is convenient to introduce it separately from the proof. 

\begin{dfn}[Hyperbolic sector]
Let $x\in\lm{K}$ and let $y,z$ be points on a hyperbola with centre $x$ (either both in the future or both in the past). 
That is, either $\tau(x,y)=\tau(x,z) > 0$ or $\tau(y,x)=\tau(z,x) >0$. 
Then the (closure of the) region bounded by $[x,y], [x,z]$ and the arc of the hyperbola between $y$ and $z$ will be called the \emph{hyperbolic sector between $[x,y]$ and $[x,z]$. }
\end{dfn}

\begin{lem}[Anti-Lipschitz property of straightened Alexandrov situation]
\label{lem: anti lip alexlem}
Let $x_1 \ll x_2 \ll x_3 \ll x_4$ be four points in $\lm{K}$ (we may refer to such a constellation as a timelike quadrilateral) such that $x_2$ and $x_4$ are on opposite sides of the line generated by the segment $[x_1,x_3]$, i.e.\ the triangles $\Delta(x_1,x_2,x_3)$ and $\Delta(x_1,x_3,x_4)$ do not overlap. 
Assume the quadrilateral is concave at $x_3$, which can be expressed in Lorentzian terms as $\ma_{x_3}(x_1,x_2) \geq \ma_{x_3}(x_1,x_4)$. 
Let $\Delta(\bx_1,\bx_2,\bx_4)$ be a timelike triangle in $\lm{K}$ such that $\tau(\bx_1,\bx_2)=\tau(x_1,x_2), \tau(\bx_1,\bx_4)=\tau(x_1,x_4)$ and $\tau(\bx_2,\bx_4)=\tau(x_2,x_3)+\tau(x_3,x_4)$. 
Then there exists a strongly long map $\varphi$ from the filled in triangle $\Delta(\bx_1,\bx_2,\bx_4)$ into the filled in concave quadrilateral given by the vertices $x_1 \ll x_2 \ll x_3 \ll x_4$ such that the boundaries are mapped in a $\tau$-length-preserving way when viewed as timelike loops. 
\end{lem}

\begin{proof}
Let $\bx_3 \in [\bx_2,\bx_4]$ be such that $\tau(\bx_2,\bx_3)=\tau(x_2,x_3)$ (and hence also $\tau(\bx_3,\bx_4)=\tau(x_3,x_4)$). 
In typical fashion of Alexandrov's Lemma, we first show $\tau(x_1,x_3) \geq \tau(\bx_1,\bx_3), \ma_{x_1}(x_2,x_3) \geq \ma_{\bx_1}(\bx_2,\bx_3), \ma_{x_1}(x_3,x_4) \leq \ma_{\bx_1}(\bx_3,\bx_4), \ma_{x_2}(x_1,x_3)\geq\ma_{\bar x_2}(\bar x_1,\bar x_3)$ and $\ma_{x_4}(x_1,x_3)\leq\ma_{\bar x_4}(\bar x_1,\bar x_3)$. 
By the reverse triangle inequality, we have $\tau(x_2,x_4) \geq \tau(x_2,x_3)+\tau(x_3,x_4)=\tau(\bx_2,\bx_4)$. 
Via Law of Cosines Monotonicity, this yields 
\begin{align*}
\ma_{x_1}(x_2,x_3) + \ma_{x_1}(x_3,x_4) & = \ma_{x_1}(x_2,x_4) \\
& \leq \ma_{\bx_1}(\bx_2,\bx_4) = \ma_{\bx_1}(\bx_2,\bx_3) + \ma_{\bx_1}(\bx_3,\bx_4) \, .
\end{align*}  
By definition of hyperbolic angles, we have $\ma_{\bx_3}(\bx_1,\bx_2)=\ma_{\bx_3}(\bx_1,\bx_4)$. 
If we were to have $\tau(x_1,x_3) < \tau(\bx_1,\bx_3)$, then the Law of Cosines Monotonicity applied to the sub triangles containing $x_3$ and $\bx_3$ together with the concavity of the quadrilateral would yield 
\begin{equation*}
\ma_{\bx_3}(\bx_1,\bx_2) > \ma_{x_3}(x_1,x_2) \geq \ma_{x_3}(x_1,x_4) > \ma_{\bx_3}(\bx_1,\bx_4) = \ma_{\bx_3}(\bx_1,\bx_2) \, ,
\end{equation*}
a contradiction. 
From the inequality $\tau(x_1,x_3) \geq \tau(\bx_1,\bx_3)$ we further obtain $\ma_{x_1}(x_2,x_3) \geq \ma_{\bx_1}(\bx_2,\bx_3)$ and $\ma_{x_1}(x_3,x_4) \leq \ma_{\bx_1}(\bx_3,\bx_4)$ as well as $\ma_{x_2}(x_1,x_3)\geq\ma_{\bar x_2}(\bar x_1,\bar x_3)$ and $\ma_{x_4}(x_1,x_3)\leq\ma_{\bar x_4}(\bar x_1,\bar x_3)$ by Law of Cosines Monotonicity. 

The inequalities $\ma_{x_1}(x_2,x_4) \leq \ma_{\bx_1}(\bx_2,\bx_4), \ma_{x_2}(x_1,x_3)\geq\ma_{\bar x_2}(\bar x_1,\bar x_3)$ and $\ma_{x_4}(x_1,x_3)\leq\ma_{\bar x_4}(\bar x_1,\bar x_3)$ imply that one can place isometric copies of the sub triangles $\Delta(x_1,x_2,x_3)$ and $\Delta(x_1,x_3,x_4)$ inside $\Delta(\bx_1,\bx_2,\bx_4)$ attached at the sides $[\bx_1,\bx_2]$ and $[\bx_1,\bx_4]$, respectively, such that they do not overlap. 
Denote the vertex corresponding to $x_3$ that belongs to the isometric copy of $\Delta(x_1,x_2,x_3)$ by $\bx_3'$ and the one belonging to the isometric copy of $\Delta(x_1,x_3,x_4)$ by $\bx_3''$. 
Denote the convex hulls of those isometric copies by $T_2$ and $T_4$, respectively. 
In other words, $T_2$ is the `filled in' triangle $\Delta(\bx_1,\bx_2,\bx_3')$ and $T_4$ is the filled in triangle $\Delta(\bx_1,\bx_3'',\bx_4)$. 

We will now build three hyperbolic sectors. 
Denote by $H_1$ the hyperbolic sector between $[\bx_1,\bx_3']$ and $[\bx_1, \bx_3'']$, by $H_2$ the hyperbolic sector between $[\bx_2,\bx_3]$ and $[\bx_2,\bx_3']$ and by $\tilde H_4$ the hyperbolic sector between $[\bx_3,\bx_4]$ and $[\bx_3'',\bx_4]$. 
Note that due to the nature of the construction, the hyperbola associated to $\tilde H_4$ will meet the interior of the triangle $\Delta(\bx_1,\bx_2,\bx_3')$ and hence $\tilde H_4$ overlaps with the other two. 
We set $H_4 = \tilde H_4 \setminus (H_1 \cup H_2)$ and allow ourselves to refer to this set as a hyperbolic sector too. 
This is to ensure that we have three disjoint sets (up to the shared boundary $[\bx_1,\bx_3']$ by $H_1$ and $H_2$), which will make the following arguments a bit easier. 
See \Cref{fig: anti-lipschitz alexlem} for a visual overview of this construction. 
\begin{figure}
\begin{center}
\definecolor{darkgreen}{rgb}{0,0.4,0.2}
\begin{tikzpicture}
\draw  (0.31172954031303624,2.9046166811518455)-- (3.0249581094073865,7.100043866570487);
\draw  (3.0249581094073865,7.100043866570487)-- (3,0);
\draw  (3,0)-- (0.31172954031303624,2.9046166811518455);
\draw  (0.31172954031303624,2.9046166811518455)-- (1.8932275251665471,4.634106743596759);
\draw  (1.8932275251665471,4.634106743596759)-- (3,0);
\draw  (3,0)-- (2.4462664677109376,4.53394098161647);
\draw  (2.4462664677109376,4.53394098161647)-- (3.0249581094073865,7.100043866570487);
\draw [dotted] (0.9052482898024268,3.822366377962175)-- (3,0);
\draw [samples=50,domain=-0.75:0.1,rotate around={90:(0.31172954031303624,2.9046166811518455)},xshift=0.31172954031303624cm,yshift=2.9046166811518455cm,color=blue] plot ({0.7*(1+(\x)^2)/(1-(\x)^2)},{0.7*2*(\x)/(1-(\x)^2)});
\draw [samples=50,domain=-0.19:0.19,rotate around={90:(3,0)},xshift=3cm,yshift=0cm,color=darkgreen] plot ({4.5*(1+(\x)^2)/(1-(\x)^2)},{4.5*2*(\x)/(1-(\x)^2)});
\draw [samples=50,domain=-0.39:0.19,rotate around={90:(3.0249581094073865,7.100043866570487)},xshift=3.0249581094073865cm,yshift=7.100043866570487cm,color=red] plot ({2.5*(-1-(\x)^2)/(1-(\x)^2)},{2.5*(-2)*(\x)/(1-(\x)^2)});
\begin{scriptsize}
\coordinate [circle, fill=black, inner sep=0.7pt, label=270: {$x_1$}] (x1) at (0-1,0);
\coordinate [circle, fill=black, inner sep=0.7pt, label=180: {$x_2$}] (x2) at (-2.3663018065581007-1,2.609479687546951);
\coordinate [circle, fill=black, inner sep=0.7pt, label=135: {$x_3$}] (x3) at (-0.5713218393200059-1,4.536122644294794);
\coordinate [circle, fill=black, inner sep=0.7pt, label=90: {$x_4$}] (x4) at (-0.002578052400444981-1,7.1000004680530955);
\coordinate [circle, fill=black, inner sep=0.7pt, label=270: {$\bar x_1$}] (bx1) at (3,0);
\coordinate [circle, fill=black, inner sep=0.7pt, label=180: {$\bx_2$}] (bx2) at (0.31172954031303624,2.9046166811518455);
\coordinate [circle, fill=black, inner sep=0.7pt, label=90: {$\bar x_4$}] (bx4) at (3.0249581094073865,7.100043866570487);
\coordinate [circle, fill=black, inner sep=0.7pt, label=135: {$\bar x_3$}] (bx3) at (0.9052482898024268,3.822366377962175);
\coordinate [circle, fill=black, inner sep=0.7pt, label=90: {$\bar x_3'$}] (bx3') at (1.8932275251665471,4.634106743596759);
\coordinate [circle, fill=black, inner sep=0.7pt, label=45: {$\bar x_3''$}] (bx3'') at (2.4462664677109376,4.53394098161647);
\coordinate [label={[darkgreen]90: {$H_1$}}] (h1) at (2.4,3.2);
\coordinate [label={[blue]90: {$H_2$}}] (h2) at (0.3,3.2);
\coordinate [label={[red]225: {$H_4$}}] (h4) at (2.8,6.1);
\coordinate [label={225: {$T_2$}}] (t2) at (2,3.5);
\coordinate [label={225: {$T_4$}}] (t4) at (3.1,4);
\end{scriptsize}
\draw  (x1) -- (x2) -- (x3);
\draw (x1) -- (x3) -- (x4) -- (x1);
\end{tikzpicture}
\end{center}
\caption{The three hyperbolic sectors $H_1, H_2$ and $H_4$ and the isometrically copied triangles $T_2$ and $T_4$ are inscribed in the straightened out Alexandrov configuration on the right hand side.}
\label{fig: anti-lipschitz alexlem}
\end{figure}
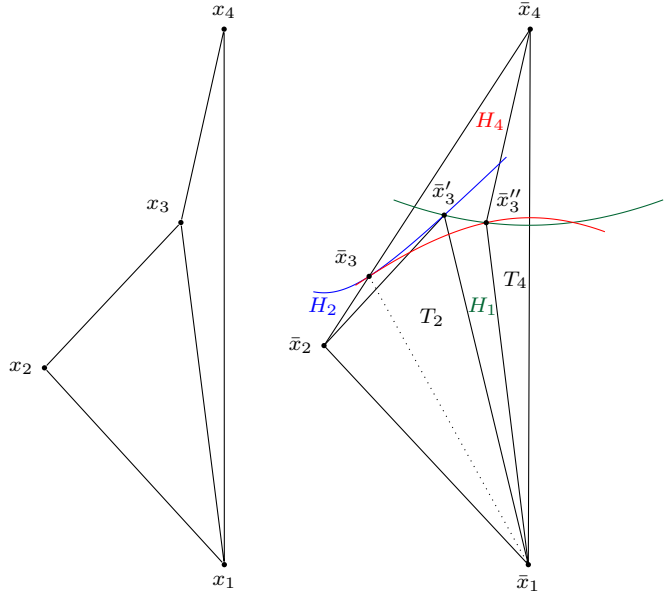
Observe that the two isometric copies of the sub triangles together with $H_1, H_2$ and $H_4$ cover the whole triangle $\Delta(\bx_1,\bx_2,\bx_4)$. 

We now define the long map $\varphi$ as follows. 
Any point in $T_2$ or $T_4$ will be mapped to its corresponding point in the filled in triangle $\Delta(x_1,x_2,x_3)$ or $\Delta(x_1,x_3,x_4)$ using the isometry, respectively. 
Any point in $H_1, H_2$ or $H_4$ will be mapped onto the associated segment from the centre to $x_3$ in the concave quadrilateral at the appropriate distance from the centre. 
For example, if $p\in H_1$ and $\tau(\bx_1,p) = r$ then $p$ will be mapped to a point $\tilde p$ on $[x_1,x_3]$ such that $\tau(x_1,\tilde p)=r$. 

We check that this assignment is strongly long by distinguishing several cases. 
Some arguments are made clearer by viewing $\varphi(p)$ as being in the appropriate part of $\Delta(\bx_1,\bx_2,\bx_4)$, hence we will do so. 
If both points are in $T_2$ or both in $T_4$, then the assignment is $\tau$-isometric. 
If they are both in either $H_1, H_2$ or $H_4$ then applying \Cref{lem: Maximising distances in model space}(ii) once yields the claim. 
Now suppose one point is in a triangle and the other one in an adjacent hyperbolic sector. 
Assume first $p \in T_2$ and $r \in H_1$. 
Then also in this case \Cref{lem: Maximising distances in model space}(ii) gives the claim. 
If $p \in T_2$ and $r \in H_2$, then we apply \Cref{lem: Maximising distances in model space}(i). 
By symmetry, we can apply the same argument if one point is in $T_4$ and the other one in $H_1$ or $H_4$. 
The same line of thought is also valid when the two points are in $T_2$ and $T_4$, respectively. 

Now suppose $p \in H_2$ and $r \in H_4$, then $p \leq r$. 
Consider the segment from $r$ to $\bx_2$. 
Denote by $p_1$ the point on $[\bx_2,r]$ with the same $\tau$-distance to $\bx_2$ as $p$, and by $r_1$ the point on $[p_1,\bx_4]$ with the same $\tau$-distance to $\bx_4$ as $r$. 
Continue this iteratively to obtain a sequence of points $p_n \in [\bx_2,r_{n-1}]$ and $r_n \in [p_n,\bx_4]$ (by setting $r_0=r$, if you will). 
Clearly, both sequences converge to points on $[\bx_2,\bx_3] $ and $[\bx_3,\bx_4]$ with the appropriate $\tau$-distances from $\bx_2$ and $\bx_4$, respectively, and the sequence is increasing in $\tau$, i.e., $\tau(p_n,r_n) \leq \tau(p_{n+1},r_{n+1})$, by applying \Cref{lem: Maximising distances in model space} repeatedly. 
In particular, $p_n\to\varphi(p)$ and $r_n\to\varphi(r)$ and thus $\tau(p_n,r_n) \to \tau(\varphi(p),\varphi(r))$. 
Moreover, since we started with $p \leq r$ and $\leq$ is closed, we also end up with $\varphi(p) \leq \varphi(q)$. 
Similarly, if $p \in H_1$ and $r \in H_4$, then $p \leq r$, and we can proceed in the same way, except that in this case the sequences $p_n$ and $r_n$ would eventually converge to points on $[\bx_1,\bx_4]$, which is not where $\varphi(p)$ and $\varphi(q)$ lie. 
However, we simply stop the procedure as soon as we would `cross' into $T_4$. 
More precisely, say in the $i$-th step $p_i$ would cross into $T_4$. 
Then we instead choose as $p_{i}$ the point \emph{on} $[\bx_1,\bx_3'']$ with the correct distance, which is exactly $\varphi(p)$. 
Applying a similar rule to $r_i$ yields $\varphi(r)$ and hence the claim.
  
If $p \in H_1, r \in H_2$ and $p \leq r$, then by treating $\varphi(r)$ as being in (the boundary of) $T_2$, we obtain $\tau(p,r) \leq \tau(p,\varphi(r))$ via \Cref{lem: Maximising distances in model space}(i). 
Then \Cref{lem: Maximising distances in model space}(ii) can be applied to $p$ to obtain the desired inequality. 
An analogous argument holds if $r \leq p$.  

Finally, assume one point is in a triangle and the other in a sector not adjacent to that triangle. 
Say $p \in T_2, r \in H_4$ and $p \leq r$. 
Consider the connecting segment $[p,r]$ and denote by $q$ the point where it leaves the triangle $T_2$. 
Clearly, $\tau(p,r)=\tau(p,q) + \tau(q,r) = \tau(\varphi(p),\varphi(q)) + \tau(q,r)$, so it suffices to show $\tau(q,r) \leq \tau(\varphi(q),\varphi(r))$. 
If $q \in [\bx_1,\bx_3']$ then there is a corresponding point $q' \in [\bx_1, \bx_3'']$ (in the sense of having the same $\tau$-distance from $\bx_1$ such that $\varphi(q)=\varphi(q')$), and we find ourselves in the previous case of one point in (the boundary of) a triangle and the other in the adjacent sector. 
If $q \in [\bx_2,\bx_3']$, then we are in the case of $q \in H_2, r \in H_4$. 
Thus, $\tau(q,r) \leq \tau(\varphi(q),\varphi(r))$ and in total 
\begin{align*}
\tau(p,r)& =\tau(p,q) + \tau(q,r) = \tau(\varphi(p),\varphi(q)) + \tau(q,r) \\
& \leq \tau(\varphi(p),\varphi(q)) + \tau(\varphi(q),\varphi(r)) \leq \tau(\varphi(p),\varphi(r)) \, .
\end{align*}

This covers all possible cases. All of this also works for the implication of the causal relation. 
\end{proof}

\begin{rem}[On generalisations of Alexandrov configurations]
\label{rem: null anti lip alexlem} 
The previous lemma can be generalised by allowing some of the sides to be null: 
(i) If $x_1 \leq x_2$ are null related, we consider $x_2^t\in[x_2,x_3]$ (with $x_2^t\to x_2$ as $t\to0$, $x_2^t\neq x_2$ for $t>0$). 
Apply \Cref{lem: anti lip alexlem} to $x_1\ll x_2^t\ll x_3\ll x_4$, obtaining configurations $\Delta(\bx_1,\bx_2^t,\bx_4)$ including $\bar x_3^t,\bar x_3'^t,\bar x_3''$ and corresponding long maps $\varphi^t$. 
Let us assume $[\bx_1,\bx_4]$ to be fixed independent of $t$, then also $\bar x_3''$ is independent of $t$. 

We consider the limits $\bx_2=\lim_{t\to0}\bx_2^t$, which will be on a null geodesic emanating from $\bx_1$, as well as the limits $\bar x_3=\lim_{t\to0}\bar x_3^t$ and $\bar x_3'=\lim_{t\to0}\bar x_3'^t$.

Then $\Delta(\bx_1,\bx_2,\bx_4)$ has the side-lengths prescribed in the lemma. 
Also, as all sequence members had that property, we have that $T_2,T_4$ are contained in $\Delta(\bx_1,\bx_2,\bx_4)$ and do not overlap. 
Observe that all hyperbolic sectors in the limit configuration exist:
as $0=\tau(x_1,x_2)<\tau(x_1,x_3)$ and $\bx_3',\bx_3''$ were below $[\bx_2,\bx_4]$, we have that $\ma_{\bx_1}(\bx_3',\bx_3'')<\infty$. 
Since $0<\tau(x_2,x_3)$, we have that $\ma_{\bx_2}(\bx_3,\bx_3')<\infty$. 
Thus, all hyperbolic sectors $H_1,H_2,H_4$ are well-defined and have finite angles. Thus, we can construct the map $\varphi$ as in \Cref{lem: anti lip alexlem} and observe that the individual arguments for the longness of $\varphi$ do not rely on the timelikeness of $[x_1,x_2]$.

(ii) If $x_2 \leq x_3$ are null related, the sector $H_2$ ceases to exist, but otherwise the construction can be carried out in the same way. The only argument which needs slight modification is if $p \in T_2, r \in H_4, p \leq r$, which can be further reduced to $p \in [\bx_2,\bx_3']$. Indeed, as in \Cref{lem: anti lip alexlem}, we consider the point $q$ where the segment $[p,r]$ leaves $T_2$. If this happens via $[\bx_1,\bx_3']$, then we can proceed as in the case of $p \in H_1, r \in H_4$, which remains unchanged to the timelike version. If this happens via $[\bx_2,\bx_3']$, we can assume $p \in [\bx_2,\bx_3']$, and simply compute 
\begin{align*}
\tau(\varphi(p),\varphi(r)) & \geq \tau(\varphi(p),x_3) + \tau(x_3,\varphi(r)) = \tau(x_3,\varphi(r)) \\ 
& = \tau(x_3,x_4) - \tau(\varphi(r),x_4) = \tau(\bx_3,\bx_4) - \tau(r,\bx_4) \\
& \geq \tau(p,\bx_4) - \tau(r,\bx_4) \geq \tau(p,r) \, .
\end{align*}

(iii) Finally, note that if $x_3\leq x_4$ are null related, the assumption that the triangles are required to not overlap already forces the quadrilateral to be convex.
\end{rem}

The next lemma is obtained by iterative applications of the previous one.

\begin{lem}[Majorisation of polygonal paths in $\lm{K}$]
\label{lem: polygonal majorisation model space}
Let $x\ll y$ and let $\alpha,\beta$ be two future-directed piecewise timelike geodesics in $\lm{K}$ from $x$ to $y$ forming a timelike loop $C$.
Then there exists a timelike loop $\bar C=(\bar\alpha,\bar\beta)$ consisting of future-directed piecewise timelike geodesics which form a convex polygon and a strongly long map $f: R(\bar C) \to R(C)$ such that $C$ and $\bC$ are mapped in a $\tau$-length-preserving way.
\end{lem} 

\begin{proof}
First note that $C$ is already a polygon, it might just be not convex. If $C$ forms a convex polygon, then $\bar C = C, f= \id$ and we are done. 
Otherwise, we proceed by induction on the total number $n$ of 
break points of $C$. 
If $n=0$, the loop is degenerate (and in particular, convex). 
If $n=1$, then $C$ is a timelike triangle, which is clearly convex. 
If $n=2$, then we distinguish whether both breakpoints lie on one curve or not. 
If, say, $\alpha$ has two breakpoints and $C$ is not convex, we are in the situation of \Cref{lem: anti lip alexlem} and can proceed using that construction.
If both $\alpha$ and $\beta$ have a single breakpoint, call them $p_{\alpha}$ and $p_{\beta}$, non-convexity means that $p_{\alpha}$ and $p_{\beta}$ lie on the same side of $[x,y]$. 
Then the loop given by $\bar \alpha=\alpha$ and $\bar\beta$ being the (Lorentz-) reflection of $\beta$ through $[x,y]$ bounds a convex region. 
The map $f$ that restricts to the identity on $\Delta(x,p_{\alpha},y)$ and to the reflection $\Delta(x,\bar p_{\beta}, y) \mapsto  \Delta(x, p_{\beta}, y)$ is easily seen to be strongly long by \Cref{lem: Maximising distances in model space}(i). 

Let $n \geq 3$, in which case $C$ consists of $n$ triangles (by connecting each breakpoint to $x$). 
Without loss of generality, assume that $\alpha$ has at least one breakpoint. 
By removing the triangle $T_1$ that contains the lowest (with respect to ordering of the parameter values) breakpoint on $\alpha$, which we denote by $p_1$, one ends up with a timelike loop $C'$ consisting of $n-1$ triangles. 
In particular, the loop is made up of $n+2$ distance realisers. 
Without loss of generality we assume its vertices are breakpoints of $\alpha$. 

For $C'$ we apply the induction hypothesis to obtain a convex polygon $\bar C'$ and a strongly long map $f_n$ between the corresponding regions. 
The first step is to show that the piecewise distance realisers that $\bar C'$ is made up of actually correspond to those of $C'$. 
If that is the case, $\bar C'$ consists of $n-1$ triangles (or less, some of them might be straightened out, so to say, in the spirit of \Cref{lem: anti lip alexlem}). 
To this end, assume $p'$ and $q'$ are points on a segment of, say, $\alpha'$ (which is the concatenation of $[x,p_1]$ and the obvious restriction of $\alpha$ corresponding to $C'$). Setting $p'=\alpha'(s), q' = \alpha'(t), \bp' = \bar \alpha'(s), \bq'=\bar \alpha'(t)$, we simply compute 
\begin{equation*}
L_{\tau}(\alpha'|_{[s,t]})=\tau(p',q') \geq \tau(\bp',\bq')\geq L_{\tau}(\bar \alpha'|_{[s,t]}) = L_{\tau}(\alpha'|_{[s,t]}) \, ,
\end{equation*}
showing that $\bar \alpha'$ consist of segments corresponding to those of $\alpha'$. 
In particular, $\bar\alpha'$ starts with a distance realiser of the same length as the long side of $T_1$. 
Thus, one can isometrically attach the remaining triangle, which together with $C'$ makes up $C$, also to $\bar C'$. 
Denote this (filled in) triangle by $\bT_1$ (and attach it on the `outside' of $R(\bC')$ using a Lorentz transformation). 

We now check that $f_n: R(\bar C') \cup \bT_1 \to R(C') \cup T_1$ is strongly long, where $f_n|_{\bT_1}$ is the (unique) isometry transforming $\bT_1$ into $T_1$ accordingly. 
In particular, note that $R(C') \cup T_1=R(C)$. 
The only case of interest (up to symmetry) is $p \in \bT_1, r \in R(\bar C')$ and $p \leq r$. 
On the geodesic $[p,r]$ in $R(\bar C') \cup \bT_1$ connecting $p$ and $r$, denote by $q$ the point\footnote{Note that even though $R(\bar C') \cup \bT_1$ might not be convex, such a point $q$ always exists as we consider the intrinsic time separation function on $R(\bar C') \cup \bT_1$.} where this geodesic meets the triangle side shared by $R(\bar C')$ and $\bT_1$. 
Then 
\begin{align*}
\tau(p,r) & = \tau(p,q) + \tau(q,r) \leq \tau(f_n(p),f_n(q)) +\tau(f_n(q),f_n(r)) \\
& \leq \tau(f_n(p),f_n(r)) \, ,    
\end{align*}
where we have $\tau(p,q)=\tau(f_n(p),f_n(q))$ since $f_n$ restricts to an isometry on $T_1$ and $\tau(q,r) \leq \tau(f_n(q),f_n(r))$ by induction hypothesis. 
If $p$ and $r$ are merely null related the claim follows similarly since $f_n$ was strongly long on $R(\tC)$. 

If $R(\bar C') \cup \bT_1$ is convex, the proof is finished. 
We are left to proceed where this is not the case.
If $R(\bC') \cup \bT_1$ is itself not convex, this means that $\bT_1$ together with the adjacent triangle from $R(\bC')$, call it $\bT_2$, is not convex, since $R(\bC')$ is supposed to be convex by induction hypothesis and by construction the causality forces the configuration to be convex at $x$. 
These two triangles considered isolated form a configuration with two breakpoints, i.e., either as in \Cref{lem: anti lip alexlem} or like the configuration outlined at the start of this proof. In either case, we obtain a triangle $\bT'$ and an associated strongly long map $f':\bT' \to \bT_2 \cup \bT_1$. 
Consider now a new loop $\bC''$ arising from $\bC'$ by deleting $\bT_2$ and instead adding $\bT'$. 
By combining $f_n$ and $f'$ appropriately we get a long map $R(\bC'') \to R(\bC') \cup \bT_1$. 
The composition of (strongly) long maps is again (strongly) long. 
Now it could be that we need to replay the same game, i.e., $\bC''$ might not be convex. 
This time, however, $\bC''$ consists of at most $n-1$ triangles and hence we can apply the induction hypothesis directly. 
Indeed, in the previous case of $R(\bC') \cup \bT_1$ we could not rule out that we are dealing with $n$ triangles in which case we could not apply the induction hypothesis right away.
\end{proof}

The following statement deals with preparation concerning the theorem outside the model space. 

\begin{lem}[Geodesic surface spanned by a curve is compact]
\label{lem: Geodesic surface spanned by a curve is compact}
Let $X$ be a strongly causal \LpLS with curvature bounded from above by $K$. 
Let $\alpha$ be a timelike curve from $O$ to $z$ in a ($\leq K$)-comparison neighbourhood $U$. 
Then the set of points $F= \alpha\cup \bigcup_{s<t}[\alpha(s),\alpha(t)] $ is compact.
\end{lem}
\begin{proof}
First observe that $F \subseteq U$. 
Indeed, geodesics in $U$ are unique and $\alpha$ is in $U$ by assumption, hence $[\alpha(s),\alpha(t)](r) \in U$ for all $s,t,r$. 
Let $(x_n)_{n \in \N}$ be a sequence of points in $F$, i.e., $x_n=[\alpha(r_n), \alpha(s_n)](t_n)$ for some $r_n, s_n, t_n \in [0,1]$. 
From these sequences of parameter values we can extract converging subsequences which we will not relabel. 
That is, we can assume $r_n\to r,s_n\to s$ and $t_n\to t$. 
We claim that $x_n\to x:=[\alpha(r),\alpha(s)](t)$.  
Set $a=\alpha(r), a_n=\alpha(r_n), b=\alpha(s)$ and $b_n=\alpha(s_n)$. 
In particular, $a_n \to a$ and $b_n \to b$. 

Suppose first that $0<t<1$ and $r<s$. 
Clearly, there have to exist subsequences of $r_n$ and $s_n$ with `constant relation' to the limit point. 
That is, we can assume that $r\leq r_n$ and $s_n\leq s$, the other cases are analogous. 
Consider comparison triangles $\Delta(\bar a,\bar a_n,\bar b)$ for $\Delta(a,a_n,b)$ and $\Delta(\bar a_n,\bar b_n,\bar b)$ for $\Delta(a_n,b_n,b)$. 

Since $\tau$ is continuous, we have $\tau(a_n,b) \to \tau(a,b)>0, \tau(a_n,b_n) \to \tau(a,b)>0$ and $|\tau(a_n,b_n)-\tau(a_n,b)| \to 0$. 
An elementary calculation involving the Law of Cosines hence yields that $\ma_{\bb}(\ba,\ba_n) \to 0$ and $\ma_{\ba_n}(\bb_n,\bb) \to 0$. 
By this, we get that $\bar a_n\to\bar a$ and $\bar b_n\to\bar b$. 
Using a standard argument involving the bi-exponential map, we already get the  pointwise convergence of $\bx_n=[\ba_n,\bb_n](t_n)\to [\ba,\bb](t)$. 
As a consequence, we can find arbitrarily small diamonds centred around $\bx$ with endpoints on $[\ba,\bb]$ that contain all but finitely many $\bx_n$, and it is this formulation which will give the claim. 
Indeed, for all $\varepsilon > 0$ we find $N \in \N$ such that the following holds for all $n \geq N$: 
\begin{align*}
\bx_n & \in I([\ba_n,\bb](t_n-\varepsilon),[\ba_n,\bb](t_n+\varepsilon)) \subseteq J([\ba_n,\bb](t_n-\varepsilon),[\ba_n,\bb](t_n+\varepsilon)) \\ 
& \subseteq I([\ba,\bb](t-2\varepsilon),[\ba,\bb](t+2\varepsilon)) \, .
\end{align*} 
In particular, we have 
\begin{equation*}
[\ba,\bb](t-2\varepsilon) \ll [\ba_n,\bb](t_n-\varepsilon) \ll \bx_n \ll [\ba_n,\bb](t_n+\varepsilon) \ll [\ba,\bb](t + 2\varepsilon) \, .
\end{equation*}
By upper curvature bounds, cf.\ \Cref{rem: one-sided comparison}, we hence infer that $[a_n,b](t_n-\varepsilon) \ll x_n \ll [a_n,b](t_n+\varepsilon)$ and similarly that $[a,b](t-2\varepsilon) \ll y \ll [a,b](t+2\varepsilon)$ for all $y \in I([a_n,b](t_n-\varepsilon),[a_n,b](t_n+\varepsilon))$. 
In particular, $x_n \in I([a_n,b](t_n-\varepsilon),[a_n,b](t_n+\varepsilon))$ and since this is a neigbourhood basis (indexed by $\varepsilon$) of $x$ by the strong causality of $X$, we conclude $x_n \to x$. 

In the case of $t=0$ and $r \neq 0$ or $t=1$ and $s\neq 1$, the above argument needs to be slightly adapted, as one of the sequences of governing points does not exist anymore (it would `overshoot' the geodesic). 
Say $t=0$ and $r \neq 0$, then $[\ba,\bb](-2\varepsilon)$ is not a valid description of a point in our configurations. 
However, as $r \neq 0$, we can, depending on the relation between $r_n$ and $r$, choose one of these to be the bottom governing points of a neighbourhood basis of $\bx$. 
More precisely, if there exists a subsequence such that $r < r_n$, then $I(\bar \alpha(r),[\ba,\bb](t+2\varepsilon))$ is a neigbourhood basis (indexed by $\varepsilon$) of $\bx$. 
If there is a subsequence such that $r_n \ll r$, then $I(\bar \alpha(r_n),[\ba,\bb](t+2\varepsilon))$ is a neigbourhood basis (indexed by $n$ and $\varepsilon$) of $\bx$. 

Only in the cases $r=t=0$ or $s=t=1$ we need to find governing points outside of our configuration. 
Say $r=t=0$. 
Note that in this case we have $O=x \leq x_n$ for all $n$ and by above arguments even $x_n \ll [O,b](2\varepsilon)$ for suitable choices of $\varepsilon$ and $n$. 
By strong causality, for any neighbourhood $V$ of $O$ we find $O_-^i$ and $O_+^i, i=1,\ldots, n$ such that $\cap_i I(O_-^i,O_+^i) \subseteq V$. 
Then $\cap_i I(O_-^i,[O,b](2\varepsilon))$ is a neighbourhood basis (indexed by $\varepsilon$) of $O$ that contains $x_n$. 
This finishes the proof. 
\end{proof}
\begin{lem}[Long skeleton]
\label{lem: long skeleton}
Let $U$ be a $(\leq K)$-comparison neighbourhood in a \LpLS $X$. 
Let $O \in U$ and let $p_1, \ldots, p_n$ in $I^+(O)\cap U$ be such that subsequent points are timelike related (not necessarily in a monotonous way). 
For each triangle $T_i=\Delta(O,p_i,p_{i+1})$, consider its comparison triangle $\bT_i=\Delta(\bar O, \bp_i, \bp_{i+1})$ arranged in such a way that $\bT_i$ and $\bT_{i+1}$ share the side $[\bar O, \bp_{i+1}]$ and no two filled in triangles overlap. 
Denote by $\psi : \bigcup_i \bT_i \to \bigcup_i T_i$ the map that sends each comparison point in a triangle to its originally corresponding point. 
Equip $\bigcup_i \bT_i$ with the restriction of the intrinsic time separation arising from the surface built from the filled in triangles. 
In other words, the distance in $\bigcup_i \bT_i$ is measured only with respect to curves that do not leave the union of the filled in triangles. 
Then $\psi$ is long. 
If $U$ is a $(\leq K)$-comparison neighbourhood in the sense of strict triangle comparison, then $\psi$ is strongly long. 
If $U$ is a $(\leq K)$-comparison neighbourhood in the sense of (strict) causal triangle comparison, then all triangles $T_i$ are allowed to be admissible causal, i.e.\ they may contain a single null side. In this case, $\psi$ restricted to the timelike sides is (strongly) long. 
\end{lem}

\begin{proof}
We write $\psi(\bx)=x$ in the typical notation of comparison points in comparison triangles. 
If $x,y \in T_i$, then $\tau(\bx,\by) \leq \tau(x,y)$ follows directly by curvature bounds. 
If $x \in T_i, y \in T_j$, suppose without loss of generality $i \leq j$ and $\bx \leq \by$. 
For $k=i+1, \ldots, j$, denote by $\bq_k$ the point\footnote{It is precisely for the existence of these points that we view $\bigcup_i \bT_i$ with this particular time separation. In particular, $[\bx,\by]$ might not be a distance realiser in the ambient $\lm{K}$.} on $[\bO,\bp_k]$ where the geodesic $[\bx,\by]$ enters from $\bT_{k-1}$ into $\bT_{k}$. 
Then 
\begin{align*}
\tau(\bx,\by) & =\tau(\bx,\bq_{i+1}) + \sum_{k=i+2}^{j-2}\tau(\bq_k,\bq_{k+1}) + \tau(\bq_j,\by) \\
& \leq \tau(x,q_{i+1}) + \sum_{k=i+2}^{j-2}\tau(q_k,q_{k+1}) + \tau(q_j,y) \leq \tau(x,y) \, , 
\end{align*}
where the first inequality holds since each of these $\tau$-distances is now taken in a single triangle and the second inequality is due to reverse triangle inequality. 

The transitivity of $\leq$ gives that the causal relation is similarly preserved. 

Finally, the argument also applies verbatim in the case where $T_i$ may have a single null side ($x$ and $y$ shall not be chosen to lie on a null side). 
\end{proof}

\begin{thm}[Reshetnyak Majorisation, Lorentzian version]
\label{thm: Reshetnyak Majorisation}
Let $X$ be a strongly causal \LpLS with curvature bounded above by $K\in\mb R$ and let $U$ be a ($\leq K$)-comparison neighbourhood. 
Let $O \ll z$ with $\tau(O,z)<D_K$ in $U$ and let $\alpha$ and $\beta$ be two future-directed timelike rectifiable curves from $O$ to $z$ in $U$ forming a timelike loop $C$. 
Then there exists a causal loop $\tC$ in $\lm{K}$ bounding a convex region and a long map $f: R(\tC) \to U$ such that $\tC$ is mapped $\tau$-length preservingly onto $C$.
\end{thm}

\begin{proof}
We first prove the theorem in the case of $\beta=[O,z]$. 
We begin with constructing a map from a convex $n$-polygon which is `almost long' and from which we will obtain the desired curve and the associated mapping in the limit. 
For now fix $n \in \N$. 
We will hide $n$ in the notation for now and only add it later for emphasis. 

\textbf{Building a composite map for fixed $n$: }Let $p_i=\alpha(\frac{i}{2^n})$ for $i=0, \ldots, 2^n$. 
Then $p_0=O$ and $p_{2^n}=z$. 
Let $T_i=\Delta(O,p_i,p_{i+1})$ for $i=1, \ldots, 2^n-1$ and let $\bT_i$ be its comparison triangle. 
Arrange the comparison triangles in such a way that $\bT_i$ and $\bT_{i+1}$ share the side $[\bO,\bp_{i+1}]$ and are on opposite sides of that segment, i.e., they do not overlap. 
This yields an $(2^n+1)$-sided polygon in $\lm{K}$ which we may view as a timelike loop and denote by $\bar C_n$. 
Recall that $R(\bar C_n)$ denotes the region bounded by this polygon. 
By $\Sigma(\bar C_n)$ we will denote the union of the segments $[p_i,p_{i+1}], i=0, \ldots, 2^n-1$ and $[O,p_i], i=2, \ldots, 2^n$. 
We will refer to this as the \emph{skeleton} of $\bar C_n$.
Further, denote by $\tilde C_n$ the convex polygon associated to $\bC_n$ that is obtained by \Cref{lem: polygonal majorisation model space} and the corresponding long map between the regions by $\varphi_n$. 

Observe that $\Sigma(\bar C_n)$ may also be viewed as a union of comparison triangles. 
From this point of view, by \Cref{lem: long skeleton}, we have that the assignment $\psi_n$ that sends each comparison point in $\Sigma_n(\bar C_n)$ (equipped with the restriction of $\tau$ coming from $R(\bar C_n)$) to its original point in $\bigcup_{i=1}^{n-1} T_i$ is (strictly) long. 

Now let us consider the map $S_n$ from $R(\bC_n)$ to its skeleton $\Sigma(\bC_n)$ defined in the following way: any point already belonging to $\Sigma(\bC_n)$ (viewed as a subset of $R(\bC_n)$) is mapped identically. 
Any point in the interior of a triangle is `projected' onto the longest side of that triangle according to its distance from $\bO$. 
That is, if $x$ is in the interior of $\bT_i$ and $\tau(\bO,x)=r$, then $S_n(x)$ is the unique point on $[\bO,\bp_{i+1}]$ with $\tau(\bO,S_n(x))=r$. 

\textbf{Estimating how far the map is from being long: }
Unfortunately, it is easily seen\footnote{For example, the $\tau$-distance between a point very close to a short side of a triangle and a timelike related point on that short side will result in a spacelike relation after applying $S_n$.} that $S_n$ is not long. 
However, we can bound the `error' the map makes, and we will later on show that this vanishes as $n \to \infty$. 
Assume $x \in \bT_i, y \in \bT_j$ and observe that 
\begin{align*}
|\ma_O(x,y)-\ma(S_n(x),S_n(y))| & = |\ma_O(x,S_n(x) - \ma_O(y,S_n(y)| \\ 
& \leq \max_{k \geq \min(i, j)} \ma_O(\bp_k,\bp_{k+1}) \eqqcolon A^n(i,j) \, . 
\end{align*}
In words, the difference of the two angles at $\bO$ between $x$ and $y$ and its image points is bounded by the maximal angle at $\bO$ of all triangles $\bT_k$ which have a higher index than $i$ or $j$.
The number of these triangles will change (and in general increase) as $n$ increases. 
Call $a = \tau(\bO,x)=\tau(\bO,S_n(x)), b = \tau(\bO,y)=\tau(\bO,S_n(y)), c=\tau(x,y), c_n=\tau(S_n(x),S_n(y)), \omega=\ma_O(x,y), \omega_n=\ma_O(S_n(x),S_n(y))$ and $B=\tau(O,z)$. 
By inserting into the Law of Cosines \Cref{prop: law of cosines monotonicity} for $\Delta(O,x,y)$ and $\Delta(O,S_n(x),S_n(y))$ and subtracting, in the case of $K=0$, two terms cancel and we get
\begin{align*}
|c^2-c_n^2| & = 2ab|\cosh(\omega)-\cosh(\omega_n)| = 
4ab\sinh\left(\frac{\omega+\omega_n}{2}\right)\sinh\left(\left|\frac{\omega-\omega_n}{2}\right|\right) \\
& \leq 4ab\cosh(\max \{\omega, \omega_n\})\sinh\left(\left|\frac{\omega-\omega_n}{2}\right|\right) \\
& = 2(a^2+b^2-\min \{c^2,c_n^2\})\sinh\left(\left|\frac{\omega-\omega_n}{2}\right|\right) \\
& \leq 4B^2\sinh\left(\left|\frac{\omega-\omega_n}{2}\right|\right)\leq 4B^2\left|\sinh\left(\frac{1}{2}A^n(i,j)\right)\right| \, .
\end{align*}
Finally, using $|c-c_n|=\frac{|c^2-c_n^2|}{(c+c_n)}\leq\frac{|c^2-c_n^2|}{c}$ we obtain 
\begin{equation}
\label{eq: definition epsilon K=0}
|c-c_n| \leq 
\frac{4B^2}{c}|\sinh(\frac{1}{2}A^n(i,j))| \, . 
\end{equation}
A slightly more tedious yet still elementary calculation yields estimates for $K \neq 0$. 
In the case of $K < 0$, without loss of generality $K=-1$, we have 
\begin{align*}
& |\cos(c)-\cos(c_n)| = \sin(a)\sin(b)|\cosh(\omega)-\cosh(\omega_n)| \\
& \leq 2\sin(a)\sin(b)\cosh(\max \{\omega, \omega_n\})\sinh\left(\left|\frac{\omega-\omega_n}{2}\right|\right) \\
& = 2\left(\min \{\cos(c),\cos(c_n)\} - \cos(a)\cos(b)\right)\sinh\left(\left|\frac{\omega-\omega_n}{2}\right|\right) \\
& \leq 4\sinh\left(\left|\frac{\omega-\omega_n}{2}\right|\right)=4\left|\sinh\left(\frac{1}{2}A^n(i,j)\right)\right| \, .
\end{align*}
Then using $|c-c_n| = 2\arcsin\left(\frac{|\cos(c)-\cos(c_n))|}{|2\sin\left(\frac{c+c_n}{2}\right)|}\right) \leq 2\arcsin\left(\frac{|\cos(c)-\cos(c_n)|}{|2\sin\left(\frac{c}{2}\right)|}\right)$, we obtain 
\begin{equation}
\label{eq: definition epsilon K<0}
|c-c_n| \leq 2\arcsin\left({\frac{2}{|\sin{(\frac{c}{2})}|}}|\sinh(\frac{1}{2}A^n(i,j)|\right) \, .
\end{equation}
In the case of $K > 0$, without loss of generality $K=1$, we have 
\begin{align*}
& |\cosh(c)-\cosh(c_n)| = \sinh(a)\sinh(b)|\cosh(\omega)-\cosh(\omega_n)| \\
& \leq 2\sinh(a)\sinh(b)\cosh(\max \{\omega, \omega_n\})\sinh\left(\left|\frac{\omega-\omega_n}{2}\right|\right) \\
& = 2\left(\cosh(a)\cosh(b) - \min \{\cosh(c),\cosh(c_n)\}\right)\sinh\left(\left|\frac{\omega-\omega_n}{2}\right|\right) \\
& \leq 4\cosh(B)^2\sinh\left(\left|\frac{\omega-\omega_n}{2}\right|\right)=4\cosh(B)^2\left|\sinh\left(\frac{1}{2}A(i,j)\right)\right| \, .
\end{align*}
Via the identity $|c-c_n|=\arsinh(\frac{|\cosh(c)-\cosh(c_n)|}{2\sinh{(\frac{c+c_n}{2})}}) \leq \arsinh(\frac{|\cosh(c)-\cosh(c_n)|}{2\sinh{(\frac{c}{2})}})$ we obtain
\begin{equation}
\label{eq: definition epsilon K>0}
|c-c_n| \leq \arsinh \left(\frac{2\cosh(B)^2}{\sinh{(\frac{c}{2})}}|\sinh(\frac{1}{2}A(i,j)|) \right) \, .
\end{equation}
We allow ourselves to simultaneously denote the right hand side in the equations \eqref{eq: definition epsilon K=0}, \eqref{eq: definition epsilon K<0} and \eqref{eq: definition epsilon K>0} by $\varepsilon_n(x,y)$. 
Note that $B$ and $c$ are constant and $\sinh, \arsinh$ and $\sin$ are continuous with $\sinh(0)=\arsinh(0)=\sin(0)=0$. 
From these calculations, we infer $\tau(x,y) \leq \tau(S_n(x),S_n(y)) + \varepsilon_n(x,y)$. 

Collecting everything we discussed so far, we will compose the maps $\varphi_n : R(\tC_n) \to R(\bC_n)$ from \Cref{lem: polygonal majorisation model space}, the projection onto the skeleton $S_n : R(\bC_n) \to \Sigma(\bC_n)$ and the `comparison map' $\psi_n:\Sigma(\bC_n) \to \bigcup_i T_i$ to obtain a map $f_n=\psi_n \circ S_n \circ \varphi_n : R(\tC_n) \to \bigcup_i T_i$ from the convex region into the skeleton of the fan associated to $C_n$ that satisfies 
\begin{equation}
\label{eq: almost longness of f}
\tau(x,y) \leq \tau(f_n(x),f_n(y)) + \varepsilon_n(\varphi_n(x),\varphi_n(y)) \, .
\end{equation}
See \Cref{fig: majorisation construction nth step} for a visualisation of the $n$-th step. 
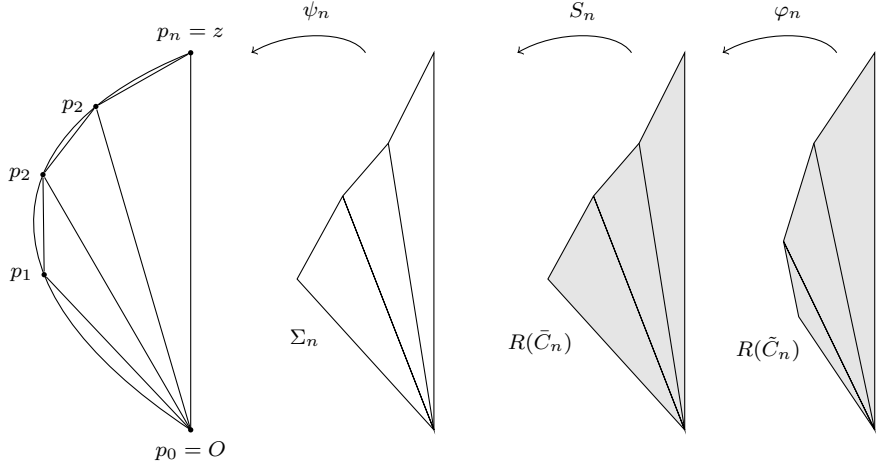
\begin{figure}
\begin{center}
\begin{tikzpicture}
\begin{scriptsize}
\draw  (0,0) -- (0,5);
\draw (0,0) .. controls (-3,2) and (-2.5,4) .. (0,5) node (p1)[circle, fill=black,inner sep=0.7pt,pos=0.35,label=left:$p_1$]{} node (p2)[circle, fill=black,inner sep=0.7pt,pos=0.6,label=left:$p_2$]{} node (p3)[circle, fill=black,inner sep=0.7pt,pos=0.8,label=left:$p_2$]{} ;;;;
\coordinate [circle, fill=black, inner sep=0.7pt, label=270: {$p_0=O$}] (O) at  (0,0);
\coordinate [circle, fill=black, inner sep=0.7pt, label=90: {$p_n=z$}] (z) at  (0,5);
\draw (0,0) -- (p1);
\draw (0,0) -- (p2);
\draw (0,0) -- (p3);
\draw (p1) -- (p2) -- (p3) -- (z);

\draw [xshift = 0.2cm] (3,0) -- (1.2,2) -- (1.8,3.1) -- (2.4,3.8) -- (3,5) -- (3,0);
\draw [xshift = 0.2cm](3,0) -- (1.8,3.1) -- (3,0) -- (2.4,3.8);
\coordinate [label=90: {$\Sigma_n$}] (S) at  (1.5,1);

\draw[xshift = 3.5cm, fill=gray, fill opacity = 0.2] (3,0) -- (1.2,2) -- (1.8,3.1) -- (2.4,3.8) -- (3,5) -- (3,0);
\draw[xshift = 3.5cm] (3,0) -- (1.8,3.1) -- (3,0) -- (2.4,3.8);
\coordinate [label=90: {$R(\bar C_n)$}] (S) at  (4.6,.9);

\draw [xshift = 6cm, fill=gray, fill opacity = 0.2] (3,0) -- (2,1.5) -- (1.8,2.5) -- (2.2,3.8) -- (3,5) -- (3,0);
\draw [xshift = 6cm] (3,0) -- (1.8,2.5) -- (3,0) -- (2.2,3.8);
\coordinate [label=90: {$R(\tilde C_n)$}] (S) at  (7.6,.8);

\draw[xshift=6cm,->] (2.5,5) .. controls (2.3,5.3) and (1.5,5.3) .. (1,5) node (p1)[pos=0.5,label=90:$\varphi_n$]{};
\draw[xshift=3.3cm,->] (2.5,5) .. controls (2.3,5.3) and (1.5,5.3) .. (1,5) node (p1)[pos=0.5,label=90:$S_n$]{};
\draw[xshift=-0.2cm,->] (2.5,5) .. controls (2.3,5.3) and (1.5,5.3) .. (1,5) node (p1)[pos=0.5,label=90:$\psi_n$]{};

\end{scriptsize}
\end{tikzpicture}
\end{center}
\caption{The $n$-th step of the iterative process outlined in the proof.}
\label{fig: majorisation construction nth step}
\end{figure}

\textbf{Letting $n$ vary: }
The next step is to send $n \to \infty$. 
Let us now label the vertices of $C_n$ as $p_i^n, i=0,\ldots, 2^n$ and the associated triangles as $T_i^n$ to signify that these are, respectively, the vertices and triangles belonging to the $n$-th step of this construction. 
By the limit curve theorem in $\lm{K}$, $\tC_n$ converges to a convex causal loop $\tC$. 
Note that we can keep $\bO$ and $\bz$ (and similarly $\tO$ and $\tz$) fixed for all $n$, say they are vertical. 
This prevents a blow up in (topological) size via, say, moving $\bz$ to infinity on a hyperbola centred at $\bO$. 
Analogously, $\bC_n$ converges to a (not necessarily convex) loop $\bC$. 
It is easily seen that (by an argument involving strong causality similar to that of \Cref{lem: Geodesic surface spanned by a curve is compact}) $C_n$ converges to the original loop $C$. 
Let us denote the vertices of $\tC_n$ by $\tp_i^n$. 
Note that some of these may not be breakpoints as the whole configuration was potentially straightened out, i.e., it might be that $\tp_i^n \in [\tp_{i-1}^n, \tp_{i+1}^n]$.
Nevertheless, we still have $\tau(p_i^n,p_{i+1}^n)=\tau(\tp_i^n,\tp_{i+1}^n)$, that is to say, the lengths of the individual segments in the $n$-th step agree. 
In particular, $L(C_n)=L(\tC_n)$. 
By nature of the convergence, we have that for every $x \in \tC$ there exists a sequence $x_n \in \tC_n$ such that $x_n \to x$. 
Similarly, for every $y \in C$ there exists a sequence $y_n \in C_n$ such that $y_n \to y$. 
Furthermore, 
points on $C_n$ and $\tC_n$ are in one-to-one correspondence to each other, with vertices $p_i^n$ corresponding to vertices $\tp_i^n$. 

\textbf{Define the limit map on the boundary: }
For $x \in \tC$, define $f(x) \coloneqq \lim_n f_n(x_n)$. 
We want to show that $f: \tC \to C$ is well-defined, i.e.\ both that it is independent of the sequence $x_n$ and that the limit exists. 
Since the sequence $(f_n(x_n))_n$ is contained in a compact set by \Cref{lem: Geodesic surface spanned by a curve is compact}, there is a converging subsequence (which we will not relabel). 
If $x_n \in [\tp_{k_n}^n, \tp_{k_n+1}^n]$ for some $(k_n)_n$, then $f_n(x_n) \in [p_{k_n}^n,p_{k_n+1}^n]$. 
Clearly, $\tau(p_{k_n}^n,p_{k_n+1}^n) \to 0$. 
Moreover, the $\tau$-length of the restriction of $\alpha$ between $p_{k_n}^n$ and $p_{k_n+1}^n$ goes to zero, which by the rectifiability of $\alpha$ implies that $(p_{k_n}^n)_n$ and $(p_{k_n+1}^n)_n$ converge to the same point on $\alpha$. 
This forces $f(x) \in C$. 
Now let $x_n'\in \tC_n$ be another sequence converging to $x$. 
Assume without loss of generality that $x_n \ll x_n'$. 
As $\tau$ is lower semi-continuous and $\tau(x,x)=0$, we infer $\tau(x_n,x_n') \to 0$. 
Moreover, if $x_n' \in [\tp_{k_n'}^n, \tp_{k'_n+1}^n]$, then $k_n\leq k'_n$ and $\tau(p_{k'_n}^n,p_{k'_n+1}^n) \to 0$. 
Thus, we compute
\begin{align*}
& \tau(\tp_{k_n}^n,x_n) + \tau(x_n,x_n') + \tau(x_n',\tp_{k_n+1'}^n)  \geq \tau(\tp_{k_n}^n,x_n) + \tau(x_n, \tp_{k+1}^n) + \ldots \\
& + \tau(\tp_{k'}^n,x_n') + \tau(x_n',\tp_{k_n+1'}) = \tau(\tp_{k_n}^n, \tp_{k+1}^n) + \ldots + \tau(\tp_{k'}^n,\tp_{k'_n+1}^n) \\
& = \tau(p_k^n,p_{k+1}^n) + \ldots + \tau(p_{k'}^n,p_{k'_n+1}^n) \geq L(\alpha|_{[s,s']}) \, ,
\end{align*}
Where $\alpha|_{[s,s']}$ denotes the restriction of $\alpha$ from $p_{k_n}^n$ to $p_{k'_n+1}^n$. 
By the rectifiability of $\alpha$ and the left most side in the above estimate going to zero, we conclude that $p_{k_n}^n$ and $p_{k'_n+1}^n$ converge to the same point. 
Since $f_n(x_n), f_n(x_n') \in J(p_{k_n}^n,p_{k'_n+1}^n)$, it follows that $\lim_n f_n(x_n)=\lim_n f_n(x_n')$. 
Finally, $f: \tC \to C$ is clearly long, as for points on the boundary of $R(\tC)$, the maps $\varphi_n$ and $S_n$ reduce to the identity and $\psi_n$ is long by \Cref{lem: long skeleton}.  
It is left to show that $L(C)=L(\tC)$.
On the one hand, $L(\tC) \leq L(C)$ follows by the longness of $f$: any partition of $\tC$ has less length than the corresponding partition of $C$. 
On the other hand, $p^n_1, \ldots p^n_{2^n}$ is a valid partition for the domain of $C$ with $\sum_i \tau(p^n_i, p^n_{i+1})=L(C_n)$, and hence by the definition of $L$ as an infimum we get $L(C)\leq \lim_n L(C_n)$. 
We also know by above considerations that $L(C_n)=L(\tC_n)$, and by upper semi-continuity of the length functional, cf.\ \cite[Proposition 3.17]{KS18}, we get $L(\tC) \geq \lim_n L(\tC_n)$. 
In total, this yields
\begin{equation}
L(\tC) \leq L(C)\leq \lim_nL(C_n)=\lim_nL(\tilde C_n) \leq L(\tC) \, .
\end{equation}

\textbf{Define limit map in interior: }
Let us extend this assignment to a long map $f:R(\tC) \to U$. 
Start out by noticing that for all $x \in R(\tC)$ there is a sequence $x_n \in R(\tC_n)$ such that $x_n \to x$. 
Let $D$ be a countable dense subset of the interior of $R(\tC)$. 
Let $x \in D$ and fix an approximating sequence $x_n \in R(\tC_n)$. 
Since $f_n(x_n) \in \alpha\cup \bigcup_{s < t}[\alpha(s), \alpha(t)]$ and this set is compact by \Cref{lem: Geodesic surface spanned by a curve is compact}, we infer the existence of a converging subsequence. Choose one such subsequence and define $f(x) = \lim_n f_n(x_n)$. 
By a classical diagonal argument, we can now assume that the whole sequence $(f_n(x_n))_n$ converges for all $x \in D$. 

We now show that $f$ is long on $D$.
Let $x,y \in D$. 
If $x=\tO$, we know that $\tau(\tO,y_n)\leq \tau(\bO,\varphi_n(y_n))=\tau(\bO,S_n(\varphi_n(y_n))\leq\tau(O,\psi_n(S_n(\varphi_n(y_n))))$, thus we have $\tau(\tO,y)\leq\tau(f(\tO),f(y))$. 
Otherwise, $\tau(x,y)=\lim_n\tau(x_n,y_n) \leq \lim_n\tau(f_n(x_n),f_n(y_n)) + \varepsilon_n(\varphi(x_n),\varphi(y_n)) = \tau(f(x),f(y) + \lim_n \varepsilon_n(\varphi(x_n),\varphi(y_n))$. 
The next step is to prove that $\varepsilon_n(\varphi(x_n),\varphi(y_n)) \to 0$. 
The idea is to show that $A^n(i_n,j_n) \to 0$, that is, the maximum angle of the triangles above $S_n(\varphi_n(x_n))$ or $S_n(\varphi_n(y_n))$ vanishes. 
Before we show this, note that $\min_{k\geq \min (i_n, j_n)}\tau(\bO,\bp_k^n)$ is bounded below, as indeed we have $\tau(\tO, x) = \lim_n \tau(\tO,x_n) \leq \lim_n \tau(\bO,\varphi_n(x_n))$ since $\varphi_n$ is long for all $n$. 
If this bound was not in place, then a potential problem may occur if $x_n$ lies in the bottom most triangle in each step of the iteration, i.e., $x_n \in T_1^n$ for all $n$, in which case we could not even bound $A^n$. 
Visually, the angles should become smaller because the triangles get narrower, but it cannot be ruled out that $T_1^n$ gets increasingly more null. 
Fortunately, for large enough $n$ there is a (positive) lower bound for $\tau(\tO,x_n)$ and hence $\varphi_n(x_n)$ cannot be in such a `decreasing' sequence of triangles. 
Thus, we aim to show
\begin{equation}
\max_{k
}(\tau(\bO,\bp_{k+1}^n)-\tau(\bO,\bp_k^n)) \to 0
\end{equation}
as $n \to \infty$, after which $A^n(i_n,j_n) \to 0$ simply follows by continuity of $\tau$ and the uniform continuity of the Law of Cosines when the adjacent side parameters are bounded below. 
To this end consider the map $g:[0,1] \to \R$ given by $g(t)=\tau(O,\alpha(t))$. 
This map is uniformly continuous since $\tau$ is continuous and the domain is compact. 
Furthermore, $p_i^n \in \alpha$ and $\tau(O,p_i^n)=\tau(\bO,\bp_i^n)$ for all $i,n$. 
The claim then follows using the definition of uniform continuity of $g$ and the way $p_i^n$ is constructed: indeed, for all $\delta >0$ there exists $N$ such that $|\frac{1}{2^n}|<\delta$ for all $n \geq N$, which is exactly the difference in parameters of $p_i^n$ and $p_{i+1}^n$, hence $|g(\frac{i+1}{2^n}) - g(\frac{i}{2^n})| = |\tau(O,p_{i+1}^n) - \tau(O,p_i^n)| < \varepsilon$. 
This allows us to conclude $A^n(i_n,j_n) \to 0$ and hence $\varepsilon_n(\varphi_n(x), \varphi_n(y)) \to 0$. 
Note that again due to the longness of $\varphi_n$ we have a positive lower bound for the denominator in \eqref{eq: definition epsilon K=0}, \eqref{eq: definition epsilon K<0} and \eqref{eq: definition epsilon K>0}. 

For $x \in \Int(R(\tC)) \setminus D$, we then define $f(x)=\lim_m f(x_m)$, where $x_m \to x$ is any sequence in $D$ approximating $x$ (for $x\in D$, we can of course set $x_m \equiv x$ and use the same formula). 
Let $x,y \in \Int(R(\tC))$, then $\tau(x,y)=\lim_m \tau(x_m,y_m) \leq \lim_m \tau(f(x_m),f(y_m)) = \tau(f(x),f(y))$ since $f$ is long on $D$. 

For, say, $x \in \Int(R(\tC))$ and $y \in \tC$ an analogous calculation holds (although one technically needs to redo some argument, they follow verbatim). 
This shows that $f$ is long on $R(\tC)$. 
As to $f(R(\tC)) \subseteq U$, recall that we said before that $f_n(x_n) \in \alpha \bigcup_{s < t}[\alpha(s), \alpha(t)]$ and this set is compact. 
This shows that even $f(R(\tC)) \subseteq \alpha \bigcup_{s < t}[\alpha(s), \alpha(t)] \subseteq U$. 

\textbf{The case of general $\beta$: }Finally, observe that the actual statement follows easily by just doing the outlined construction once each for $\alpha$ and $\beta$. 
More precisely, do the above outline constructed respectively for the loops $(\alpha, [O,z])$ and $(\beta,[O,z])$. 
This yields configurations $R(\tC_{\alpha})$ and $R(\tC_{\beta})$. 
By applying a suitable isometry to one of them, arrange them in such a way that they share $[\tO,\tz]$ and are on opposite sides of this segment. 
Now just take as $f$ the union of the maps $f_\alpha$ and $f_\beta$, which will still be long. 
The only case of interest is for points $x \in R(\tC_{\alpha}), y \in R(\tC_{\beta})$. 
Then the segment $[x,y]$ crosses $[\tO,\tz]$ in some point $q\in R(\tC_{\alpha})\cap R(\tC_{\beta})$. 
This gives
\begin{equation*}
\tau(x,y)=\tau(x,q)+\tau(q,y)\leq\tau(f_\alpha(x),f_\alpha(q))+\tau(f_{\beta}(q),f_{\beta}(y))\leq \tau(f(x),f(y)) \, .
\end{equation*}
\end{proof}

\section{Four-point condition}
In this section we introduce a formulation of upper curvature bounds via four-point configurations that are genuinely suitable for a discrete setting. 
Four-point conditions in the Lorentzian setting were originally introduced in \cite{BKR24}. 
Therein, however, only the characterisation for curvature bounded below is satisfactory for discrete spaces. 
While the upper curvature formulation presented in that work technically gets by without requiring the time separation function to be intrinsic, it still assumes a property which is at least morally close to the existence of $\tau$-midpoints (see \cite{BR23+} for a discussion about how $\tau$-midpoints relate to an intrinsic time separation function). 
This somewhat parallels the world of positive signature, where the four-point condition is most commonly the initial definition given for spaces with lower curvature bounds (going back to at least \cite{BGP92}), while CAT($k$) spaces are usually defined via classical triangle comparison. 

We refer to \cite{BKR24} for a thorough introduction to four-point configurations in Lorentzian signature. 
Let us nevertheless allow to go on a small, slightly informal, tangent, which should at least somewhat justify the namesake of the configurations outlined below, and maybe even explain the thought process of how one arrived at this characterisation in the first place. 
The following arguments are just as valid in the metric case, and in fact similarly spirited to \cite{BN08} (see also \cite{Gromov1999MetricStructures, Sat09}). 

Consider four timelike related points in a \LpLSn, say $x_1 \ll x_2 \ll x_3 \ll x_4$. 
These four points yield six $\tau$-distances in total, one for each pair of points. 
To `compare' such a configuration to a corresponding model space configuration towards the aim of obtaining a description of curvature bounds, one transports five of those six distances into the model space and hopes for an inequality on the remaining one. 
Transporting five distances to the model space essentially means to construct two comparison triangles which have one side in common, and the remaining side will be between the two vertices which are not on the shared segment. 
Further, constructing two comparison triangles along a shared side presents a choice on its own, namely if the two triangles are realised in the same half space generated by the line extending the shared segment or in different half spaces, i.e.\ if they overlap or not. 
Thus, in total we end up with 12 possible four-point configurations, and the question now becomes which of those may yield an inequality relating to curvature bounds. 
A priori, one can exclude a variety of these by simply testing two elementary properties numerically which need to be in place for curvature bounds to make sense to be expressed in that way. 
On the one hand, realising an appropriate comparison four-point configuration and looking at the resulting side to be compared, we should obtain a monotonous function in the curvature $K$ (as the model space $\lm{K}$ has curvature bounded above by any $K'\leq K$ and bounded below by any $K'\geq K$). 
The two binary parameters of upper and lower curvature bounds and the direction of the inequality between the remaining side and its corresponding distance in the model space yields four possible combinations in total. 
The type of monotonicity (increasing or decreasing) rules out two of them. 
On the other hand, any curvature inequality we obtain in this way must obviously hold true for any $\lm{K}$ (viewed as a \LpLS with curvature bounded from above and below by $K$). 
In fact, this yields a non-trivial condition for a fixed $K$.
Note that both the same-side and different-side realisation of the two triangles have the same comparison four-point configuration with one being its own comparison situation. 
The inequality between the remaining sides has to match the originally prescribed curvature inequality. 

In total, this rules out eight of the twelve possible configurations.
Two of the four which are still admissible are exactly those corresponding to curvature bounds from below, namely those where the triangles are realised on different sides and the inequality is on $\tau(x_1,x_2)$ or $\tau(x_3,x_4)$, respectively. 
In \cite{BKR24} these were called past- and future-configurations, respectively. 
These names are actually nicely compatible with the picture we try to paint here, as it is the `past most' or `future most' side that is giving comparison. 
The remaining two, via the Majorisation Theorem, turn out to successfully characterise upper curvature bounds. 

\begin{dfn}[Upper curvature bounds by four-point condition]
\label{def: CBA four point}
Let $X$ be a \LpLSn. 
An open set $U$ is called a ($\leq K$)-comparison neighbourhood in the sense of the \emph{four-point condition} if:
\begin{enumerate}
\item $\tau$ is continuous on $(U\times U)\cap \tau^{-1}([0,D_K))$, and this set is open.
\item For any $x_1\ll x_4$ in $U$, $x_2,x_3\in I(x_1,x_4)$, construct comparison triangles $\Delta(\hx_1, \hx_2, \hx_4)$ and $\Delta(\hx_1, \hx_3, \hx_4)$ realised on opposite sides of the line through $[\hx_1,\hx_4]$. 
Then $\tau(x_2,x_3) \geq \tau(\hx_2,\hx_3)$. 
\item For any $x_1 \ll x_2 \ll x_3 \ll x_4$ in $U$, construct comparison triangles $\Delta(\hx_1, \hx_2, \hx_3)$ and $\Delta(\hx_2, \hx_3, \hx_4)$ realised on the same side of the line through $[\hx_2,\hx_3]$. 
Then $\tau(x_1,x_4) \leq \tau(\hx_1,\hx_4)$. 
\end{enumerate}
\end{dfn}

\begin{dfn}[Upper curvature bounds by strict four-point condition]
\label{def: CBA strict four point}
Let $X$ be a \LpLSn. 
An open set $U$ is called a ($\leq K$)-comparison neighbourhood in the sense of the \emph{strict four-point condition} if:
\begin{enumerate}
\item $\tau$ is continuous on $(U\times U)\cap \tau^{-1}([0,D_K))$, and this set is open.
\item For any $x_1\ll x_4$ in $U$, $x_2,x_3\in (I^+(x_1)\cap J^-(x_4))\cup(J^+(x_1)\cap I^-(x_4))$ \footnote{This condition precisely requires $\Delta(x_1,x_2,x_4)$ and $\Delta(x_1,x_3,x_4)$ to be admissible causal.},  construct comparison triangles $\Delta(\hx_1, \hx_2, \hx_4)$ and $\Delta(\hx_1, \hx_3, \hx_4)$ realised on opposite sides of the line through $[\hx_1,\hx_4]$. 
Then $\hx_2 \leq \hx_3\Rightarrow x_2 \leq x_3 $ and $\tau(x_2,x_3) \geq \tau(\hx_2,\hx_3)$. 
\item For any $x_1 \leq x_2 \ll x_3 \leq x_4$ in $U$, construct comparison triangles $\Delta(\hx_1, \hx_2, \hx_3)$ and $\Delta(\hx_2, \hx_3, \hx_4)$ realised on the same side of the line through $[\hx_2,\hx_3]$. 
Then $\tau(x_1,x_4) \leq \tau(\hx_1,\hx_4)$. 
\end{enumerate}
\end{dfn}

\begin{figure}
\begin{center}
\begin{tikzpicture}
\draw (0,0) -- (0,4) -- (-1,1.5) -- (0,0) -- (0.5,3) -- (0,4);
\draw[dashed] (-1,1.5) -- (0.5,3);

\draw (2,1) -- (2,2) -- (3.1,4) -- (2,1) -- (3,-0.5) -- (2,2);
\draw[dashed] (3.1,4) -- (3,-0.5);
\begin{scriptsize}
\coordinate [circle, fill=black, inner sep=0.7pt, label=270: {$\hx_1$}] (bx1) at (0,0);
\coordinate [circle, fill=black, inner sep=0.7pt, label=90: {$\hx_4$}] (bx1) at (0,4);
\coordinate [circle, fill=black, inner sep=0.7pt, label=180: {$\hx_2$}] (bx1) at (-1,1.5);
\coordinate [circle, fill=black, inner sep=0.7pt, label=0: {$\hx_3$}] (bx1) at (0.5,3);

\coordinate [circle, fill=black, inner sep=0.7pt, label=225: {$\hx_2$}] (bx1) at (2,1);
\coordinate [circle, fill=black, inner sep=0.7pt, label=90: {$\hx_4$}] (bx1) at (3.1,4);
\coordinate [circle, fill=black, inner sep=0.7pt, label=180: {$\hx_3$}] (bx1) at (2,2);
\coordinate [circle, fill=black, inner sep=0.7pt, label=0: {$\hx_1$}] (bx1) at (3,-0.5);
\end{scriptsize}

\end{tikzpicture}
\end{center}
\caption{The two different types of four-point comparison configurations in \Cref{def: CBA four point}(ii) and (iii).}
\label{fig: CBA four point}
\end{figure}
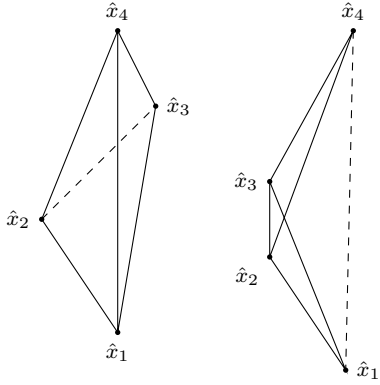

See \Cref{fig: CBA four point} for a depiction of the two different realisations of four-point configurations. 
Note that there is a small subtlety involving size bounds for (iii). 
Normally, like in \Cref{def: CBA four point}(ii) and in any other formulation from \cite{BKR24}, size bounds in the model space are guaranteed by size bounds in $X$. 
In the case of (ii) for instance, the fact that by assumption we always have $\tau(x_1,x_4) < D_K$ guarantees $\tau(\hx_1,\hx_4) < D_K$ (since they are equal) and hence the whole configuration is admissible, so to say.
However, in the case of (iii) this does not apply as it is the longest side that gives the inequality. 
More precisely, it might happen that $\tau(\hx_1,\hx_4) \geq D_K$. 
In this case, the four-point comparison configuration will satisfy $\tau(\hat x_1,\hat x_4)=+\infty$.
However, this is of little concern to us as we still have $\tau(x_1, x_4) < D_K$ by assumption and hence the inequality is trivially satisfied. 
This may be compared to how some authors assume configurations where an angle is undefined to automatically satisfy the curvature inequality, see e.g.\ \cite[Definitions 8.1 \& 9.1]{AKP24}. 

\begin{rem}[Null relations in strict four-point condition under regularity]
\label{rem: admissible causal triangles four point}
Note that if $X$ is regular, the case where $x_2\leq x_4$ are null related is not relevant in \Cref{def: CBA strict four point}(ii):
By reverse triangle inequality, we know that all $\tau(x_2,x_3),\tau(x_3,x_4),\tau(\hat x_2,\hat x_3),\tau(\hat x_3,\hat x_4)$ are zero, making the $\tau$-inequality trivial. If now $\hat x_2\leq\hat x_3$, as $\hat x_2,\hat x_3,\hat x_4$ do not lie on a distance realiser ($\hat x_2$ and $\hat x_3$ lie on different sides of the timelike segment $[\hx_1, \hx_4]$), the only possibility for this is if $\hat x_3=\hat x_4$, thus $\Delta(\hat x_1,\hat x_3,\hat x_4)$ degenerates. 
This means that $x_1,x_3,x_4$ are on a distance realiser with $[x_3,x_4]$ null. 
By regularity, this means that also $x_3=x_4$, making this implication trivial.

Analogously, the case where $x_1\leq x_3$ are null related is not relevant in \Cref{def: CBA strict four point}(ii). 
\end{rem}

We will demonstrate that \Cref{def: CBA four point} is part of the extensive family of curvature characterisations presented in \cite{BKR24} by showing it is implied by \Cref{def: triangle comparison} and implies \Cref{rem: one-sided comparison} (under the assumption of the existence of geodesics). 

\begin{prop}[Four-point condition is necessary]
\label{Four-point condition is necessary}
Let $X$ be a \LpLS and let $K \in \R$.
Let $U$ be a ($\leq K$)-comparison neighbourhood in the sense of (strict) timelike triangle comparison. 
Then $U$ is also a ($\leq K$)-comparison neighbourhood in the sense of the (strict) four-point condition.
\end{prop}
\begin{proof}
\Cref{def: CBA four point}(i) is also present in the definition of triangle comparison. 

Let $x_2,x_3 \in I(x_1,x_4)$ in $U$. 
To show \Cref{def: CBA four point}(ii), an application of \Cref{lem: long skeleton} to the quadruple $(x_1,x_2,x_3,x_4)$ yields the existence of a configuration $\hx_2, \hx_3 \in I(\hx_1, \hx_4)$ such that $\tau(\hx_i,\hx_j)=\tau(x_i,x_j)$ for $(i,j)\in\{(i,j): 1\leq i<j\leq 4\}\setminus \{(2,3)\}$ and $\tau(\hx_2,\hx_3) \leq \tau(x_2,x_3)$. 
In the spirit of the first paragraph in the proof of \Cref{lem: polygonal majorisation model space}, we can assume that the configuration is convex. 
Then $\hx_2$ and $\hx_3$ are already on different sides of $[\hx_1,\hx_4]$. 
Thus, $(\hx_1,\hx_2,\hx_3,\hx_4)$ is a comparison four-point configuration as in (ii) and we already have the desired inequality. 

Concerning \Cref{def: CBA strict four point}(ii), if $U$ is a strict comparison neighbourhood and $\Delta(x_1,x_2,x_4)$ and $\Delta(x_1,x_3,x_4)$ are admissible causal then the map in \Cref{lem: long skeleton} is strongly long, the above arguments still hold and also imply $\hx_2 \leq \hx_3 \Rightarrow x_2 \leq x_3$, as desired. 

To prove \Cref{def: CBA four point}(iii), let $x_1 \ll x_2 \ll x_3 \ll x_4$ in $U$. 
Let $\alpha$ be the concatenation of the segments $[x_1,x_2], [x_2,x_3]$ and $[x_3,x_4]$ and let $\beta$ be the segment $[x_1,x_4]$. 
Then again by \Cref{lem: long skeleton}, we obtain a configuration $\bx_1 \ll \bx_2 \ll \bx_3 \ll \bx_4$ such that $\bar\tau(\bx_i,\bx_j)=\tau(x_i,x_j)$ for $(i,j)\in\{(i,j): 1\leq i<j\leq 4\}\setminus \{(2,4)\}$ and $\bar\tau(\bx_2,\bx_4) \leq \tau(x_2,x_4)$. Note that this only holds for the intrinsic time separation function $\bar\tau$ of $\Delta(\bx_1,\bx_2,\bx_3)\cup\Delta(\bx_1,\bx_3,\bx_4)$ (which need not be convex). 
If this in not convex, we apply \Cref{lem: anti lip alexlem} to the concave quadrilateral $\bx_1 \ll \bx_2 \ll \bx_3 \ll \bx_4$ to obtain a convex quadrilateral $\hx_1 \ll \hx_2 \ll \hx_3 \ll \hx_4$ such that $\tau(\hx_i,\hx_j)=\tau(x_i,x_j)$ for $(i,j)\in\{(i,j): 1\leq i<j\leq 4\}\setminus \{(1,3), (2,4)\}$ as well as $\tau(\hx_1,\hx_3) \leq \tau(x_1,x_3)$ and $\tau(\hx_2,\hx_4) \leq \tau(x_2,x_4)$.
If this was already convex, set $\hx_i=\bx_i$, and the same inequalities hold (additionally $\tau(x_1,x_3)=\tau(\hx_1,\hx_3)$ being an equality).

By the convexity in \Cref{lem: anti lip alexlem}, we observe that already $\hx_1$ and $\hx_4$ are on the same side of $[\hx_2,\hx_3]$. 
To turn this into a valid configuration for (iii), we need to increase $\tau(\hx_1,\hx_3)$ and $\tau(\hx_2,\hx_4)$ and hope that $\tau(\hx_1,\hx_4)$ increases along the way. 
Keep all vertices but $\hx_1$ fixed and move $\hx_1$ along the hyperbola centred at $\hx_2$ until $\tau(\hx_1,\hx_3)$ reaches the appropriate distance. 
Increasing $\tau(\hx_1,\hx_3)$ while keeping the other sides fixed forces all angles in $\Delta(\hx_1,\hx_2,\hx_3)$ to increase by Law of Cosines Monotonicity. 
Since $\ma_{\hx_2}(\hx_1,\hx_4)=\ma_{\hx_2}(\hx_1,\hx_3)+\ma_{\hx_2}(\hx_4,\hx_3)$ and only the vertex $\hx_1$ is moved, this forces the angle $\ma_{\hx_2}(\hx_1,\hx_4)$ to increase. 
Now Law of Cosines Monotonicity applied to $\Delta(\hx_1,\hx_2,\hx_4)$ yields that $\tau(\hx_1,\hx_4)$ increases as well. 
Moving $\hx_4$ in the same way along the hyperbola centred at $\hx_3$ to increase $\tau(\hx_2,\hx_4)$ causes another increase of $\tau(\hx_1,\hx_4)$ in complete analogy, showing \Cref{def: CBA four point}(iii). 

Finally, to show \Cref{def: CBA strict four point}(iii), if $x_1 \leq x_2$ are null related, after possibly applying \Cref{rem: null anti lip alexlem} instead of \Cref{lem: anti lip alexlem}, we have to adapt the above argument as follows: the process of moving along a hyperbola is replaced by moving along the lightcone. Increasing $\tau(\hx_1,\hx_3)$ while keeping $\tau(\hx_1,\hx_2)=0$ fixed will move $\hx_1$ into the past, which clearly increases $\tau(\hx_1,\hx_4)$ by the reverse triangle inequality. 
The same can be done if $x_3 \leq x_4$ are null related. 
\end{proof}

\begin{prop}[Four-point condition is sufficient]
\label{Four-point condition is sufficient}
Let $X$ be a \LpLS and let $K \in \R$. 
Let $U$ be a ($\leq K$)-comparison neighbourhood in the sense of the (strict) four-point condition. Additionally, assume that for each pair $x\ll y\in U$ with $\tau(x,y)<D_K$ there is a distance realiser contained in $U$ connecting them. 
Then $U$ is also a ($\leq K$)-comparison neighbourhood in the sense of (strict) one-sided triangle comparison.
\end{prop}
\begin{proof}
Note that the existence of a realising curve for the null side of an admissible causal triangle is not at all necessary, as we never choose points on a null side.
We will show that (strict) one-sided triangle comparison holds in $U$.
Condition (i) is present in all formulations of curvature bounds, while existence of geodesics is assumed directly. 

To show the desired inequality in (strict) one-sided triangle comparison, we distinguish several cases on where the points opposing a vertex lies. 
We give the proof in the non-strict case and point out the differences. 
Let $\Delta(x,y,z)$ be a timelike triangle in $U$ and let $p$ be a point on one of its sides. 
We distinguish if $p \in [x,z]$ or $p \in [y,z]$, the latter of which is symmetric to $p \in [x,y]$. 
Note that in the strict condition, the causal implication is automatic in the second case. 

Let us first assume that $p \in [y,z]$. 
Then $x \ll y \ll p \ll z$ form a timelike four-point configuration. 
We will construct a comparison configuration as in \Cref{def: CBA four point}(iii), that is we obtain comparison triangles $\Delta(\hx,\hy,\hp)$ and $\Delta(\hy,\hp,\hz)$ situated on the same side of $[\hy,\hp]$. 
The condition in \Cref{def: CBA four point}(iii) tells us that $\tau(x,z)\leq\tau(\hx,\hz)$. 
However, since $p \in [y,z]$, the second triangle is degenerate. 
In particular, $\Delta(\hx,\hy,\hz)$ is `almost' a comparison triangle for $\Delta(x,y,z)$ with $\hp$ being a comparison point for $p$, except that $[\hx,\hz]$ is too long. 
Comparing it with the actual comparison triangle $\Delta(\bx,\by,\bz)$, two applications of the Law of Cosines Monotonicity yield the desired inequality: 
\begin{align*}
\tau(\hat x,\hat z) & \geq \tau(x, z) = \tau(\bar x, \bar z) \, ,\\
\ma_{\hat y}(\hat x,\hat p) = \ma_{\hat y}(\hat x,\hat z) & \geq \ma_{\bar y}(\bar x,\bar z) = \ma_{\bar y}(\bar x,\bar p)\, ,\\
\tau(x,p)=\tau(\hat x,\hat p) & \geq \tau(\bar x,\bar p) \, .
\end{align*}
In the strict case, note that $p \in [y,z]$ is only considered if that side is timelike. 
If $x \leq y$ are null related, then the inequalities with the angles at $y$ technically are not valid since one adjacent side is null. 
However, once can make the following easy and visual argument: when comparing $\Delta(\hx,\hy,\hz)$ and $\Delta(\bx,\by,\bz)$, arrange them in such a way that $[\by,\bz]=[\hy,\hz]$. 
In particular, $\bp=\hp$. 
Since $\tau(\bx,\bz) = \tau(x,z) \leq \tau(\hx,\hz)$, by the reverse triangle inequality this means that $\bx$ is further in the past on the null segment emanating from $\by$ than $\hx$ is. 
This immediately implies $\tau(\bx,\bp) \leq \tau(\hx,\hp)= \tau(x,p)$, again by the reverse triangle inequality. 

Now suppose $p \in [x,z]$. 
In this case, as $p,y\in I(x,z)$ we can construct a four-point comparison configuration as in \Cref{def: CBA four point}(ii). This actually precisely yields a comparison triangle for $\Delta(x,y,z)$ together with a comparison point for $p$. 
The inequalities match as well, so there is nothing more to do. 
In the strict case, the case of $x\leq y$ being null related is not interesting (as $\hat x\ll \hat p$, giving $\hat p\not\leq\hat y$). 
If $y\leq z$ are null related, then \Cref{def: CBA strict four point}(ii) yields an admissible causal triangle with the correct inequality and the correct causal implication. 
\end{proof}

\begin{rem}[Majorisation as curvature characterisation]
It is worth pointing out that in \Cref{Four-point condition is necessary} we actually did not make explicit use of triangle comparison, but rather just worked with \Cref{thm: Reshetnyak Majorisation} directly (for which, in turn, of course, we used triangle comparison throughout the proof). 
This makes it possible to view the theorem as a characterisation of upper curvature bounds in itself. 
Indeed, we have the chain of implications \Cref{def: triangle comparison} $\Rightarrow$ \Cref{thm: Reshetnyak Majorisation}/\Cref{def: CBA majorisation} $\Rightarrow$ \Cref{def: CBA four point} $\Rightarrow$ \Cref{rem: one-sided comparison} $\Leftrightarrow$ \Cref{def: triangle comparison}. 
\end{rem} 

\begin{dfn}[Curvature bounds by majorisation]
\label{def: CBA majorisation}
Let $X$ be a \LpLSn. An open subset $U$ is called a $(\leq K)$-comparison neighbourhood in the sense of \emph{majorisation} if:
\begin{enumerate}
\item $\tau$ is continuous on $(U\times U) \cap \tau^{-1}([0,D_K))$, and this set is open. 
\item $U$ is $D_K$-geodesic, i.e.\ for all $x \ll y$ in $U$ with $\tau(x,y) < D_K$ there exists a geodesic between those points inside $U$.
\item Let $x \ll y$ in $U$ with $\tau(x,y) < D_K$ and let $\alpha$ and $\beta$ be two future-directed timelike rectifiable curves from $x$ to $y$ forming a timelike loop $C$ in $U$. 
Then there exists a causal loop $\tC$ in $\lm{K}$ bounding a convex region and a long map $f:R(\tC) \to U$ such that $\tC$ is mapped $\tau$-length preservingly onto $C$. 
\end{enumerate}
\end{dfn}

Finally, analogous to \cite[Lemma~4.8]{BKR24}, we also present an angle version of the four-point condition. 

\begin{lem}[Angle version of four-point condition]
Let $X$ be a \LpLS and $U$ satisfy \Cref{def: CBA four point}(i). 

For any $x_1\ll x_4$ in $U$, $x_2,x_3\in I(x_1,x_4)$, construct comparison triangles $\Delta(\hx_1, \hx_2, \hx_4)$ and $\Delta(\hx_1, \hx_3, \hx_4)$ realised on opposite sides of the line through $[\hx_1,\hx_4]$. 
Then $\tau(x_2,x_3) \geq \tau(\hx_2,\hx_3)$ (i.e.\ \Cref{def: CBA four point}(ii)) holds if and only if $\tilde{\ma}_{x_1}(x_2,x_3)\leq\tilde{\ma}_{x_1}(x_2,x_4)+\tilde{\ma}_{x_1}(x_4,x_3)$ (or equivalently exchanging $x_1$ and $x_4$).

For any $x_1 \ll x_2 \ll x_3 \ll x_4$ in $U$, construct comparison triangles $\Delta(\hx_1, \hx_2, \hx_3)$ and $\Delta(\hx_2, \hx_3, \hx_4)$ realised on the same side of the line through $[\hx_2,\hx_3]$. 
Then $\tau(x_1,x_4) \leq \tau(\hx_1,\hx_4)$ (i.e.\ \Cref{def: CBA four point}(iii)) holds if and only if $\tilde{\ma}_{x_2}(x_1,x_4)\leq\tilde{\ma}_{x_2}(x_1,x_3)+\tilde{\ma}_{x_2}(x_3,x_4)$.
\end{lem}

\begin{proof}
In both cases, the comparison angles on the right hand side are realised in the constructed comparison triangles. 
In the first case, they are realised on opposite sides of the line through the shared segment, while in the second case they are on the same side of the shared segment.
Thus, for the first and second case, respectively, we get
\begin{align*}
\tilde{\ma}_{x_1}(x_2,x_4)+\tilde{\ma}_{x_1}(x_4,x_3)&={\ma}_{\hx_1}(\hx_2,\hx_3)\\
\tilde{\ma}_{x_2}(x_1,x_3)+\tilde{\ma}_{x_2}(x_3,x_4)&={\ma}_{\hx_2}(\hx_1,\hx_4) \, .
\end{align*}
This simplifies the right hand side of the desired inequalities. In the first case, we have $\tau(x_1,x_2)=\tau(\hx_1,\hx_2)$ and $\tau(x_1,x_3)=\tau(\hx_1,\hx_3)$, so we can apply \Cref{prop: law of cosines monotonicity}. 
This yields 
\[
\tilde{\ma}_{\hx_1}(\hx_2,\hx_3)\geq \tilde{\ma}_{x_1}(x_2,x_3) \Leftrightarrow \tau(\hx_2,\hx_3) \leq \tau(x_2,x_3) \, .
\]
The second equivalence follows similarly.
\end{proof}

\section{Outlook and applications}
A careful reader surely noticed that all preparatory results deal with strongly long maps, while the Majorisation Theorem itself only yields the existence of a long map. 
Indeed, while longness is inherited in the limit as a (non-strict) inequality, the situation becomes more complicated with a relation. 
More precisely, longness in the limit is gained by the error for the inequality going to zero. 
Being causally related or not is a binary question and hence difficult to quantify like that. 
If the maps $f_n$ along the sequence were continuous then longness improves to strong longness. 
They are, however, not continuous, a problem that is encountered very often in nonsmooth Lorentzian geometry. 

An exciting application of this work is the possible impact on causal set theory. 
In short, causal set theory is a discrete approach to quantum gravity. 
A causal set is a locally finite partially ordered set, and can be equipped with the structure of a \LpLSn, cf.\ \cite[Subsection 5.3]{KS18}. 
We refer to \cite{Sum19} for a recent overview about this topic. 
Causal sets are used to discretise and approximate spacetimes via a process that is known as \emph{Poisson sprinkling}. 
While a specific causal set obtained by sprinkling a spacetime will likely not have curvature bounds, simply due to elementary probability theory, it is reasonable to believe that if the spacetime had curvature bounds to begin with, then this is reflected in the probability with which a sprinkling has curvature bounds as well. 
Clearly, this is interesting to investigate for both upper and lower curvature bounds. 
Recently, a partial proof of the so-called \emph{Hauptvermutung} of causal set theory has been achieved using the machinery of Lorentzian length spaces \cite{MS25+}. 

The Majorisation Theorem has several applications itself. 
For example, it is used in proving a synthetic version of the famous Kirszbraun Theorem \cite{Kir34, LS97, AKP11}. 
This of course ponders the question about an extension theorem for anti-Lipschitz maps -- a Lorentzian version of the Kirszbraun Theorem. 
A McShane-type extension result for steep functions can be regarded as the first step in this direction \cite[Lemma 3.5]{Octet}. 
However, even for maps from Minkowski space to itself this is significantly more complicated. 
For a true nonsmooth version, some further machinery is required. 

Finally, we want to mention another result which makes use of the Majorisation Theorem more prominently, namely an upper bound on the length of curves in CAT($k$) spaces depending on their total curvature \cite{ML03, CM16}. 
A Lorentzian analogue to the total curvature of a curve is not an issue and hence a corresponding result about the minimal $\tau$-length of a causal curve (with, e.g., a prescribed distance between the endponts), seems feasable. 

\begin{acknowledgements}
TB was supported by the Austrian Science Fund (FWF) [Grant DOI 10.55776/PAT1996423 and 10.55776/EFP6
].

FR acknowledges the support of the European Union - NextGenerationEU, in the framework of the PRIN Project `Contemporary perspectives on geometry and gravity' (code 2022JJ8KER – CUP G53D23001810006). The views and opinions expressed are solely those of the authors and do not necessarily reflect those of the European Union, nor can the European Union be held responsible for them.

For open access purposes, the authors have applied a \href{https://creativecommons.org/licenses/by/4.0/}{Creative Commons Attribution 4.0 International} license to any author accepted manuscript version arising from this submission.
\end{acknowledgements}

\bibliographystyle{abbrv}
\addcontentsline{toc}{section}{References}
\bibliography{references} 
\end{document}